\documentclass[a4paper,12pt]{amsart}

\usepackage{tikz, hyperref, amssymb, amsmath}

\def\vert<#1>(#2,#3)[#4]{\node[circle, inner sep=0.5pt] (#1) at (#2, #3) {\small#4}}
\def\svert(#1,#2)[#3]{\vert<#3>(#1,#2)[#3]}

\newtheorem{thm}{Theorem}[section]
\newtheorem{cor}[thm]{Corollary}
\newtheorem{lem}[thm]{Lemma}
\newtheorem{prop}[thm]{Proposition}

\theoremstyle{definition}
\newtheorem{dfn}[thm]{Definition}

\theoremstyle{remark}
\newtheorem{rmk}[thm]{Remark}

\newtheorem{example}[thm]{Example}

\newcommand{\NN}{\mathbb{N}}

\newcommand{\TT}{\mathbb{T}}
\newcommand{\ZZ}{\mathbb{Z}}

\newcommand{\Bb}{\mathcal{B}}
\newcommand{\Gg}{\mathcal{G}}
\newcommand{\Hh}{\mathcal{H}}
\newcommand{\Kk}{\mathcal{K}}

\newcommand{\Mm}{\mathcal{M}}

\newcommand{\id}{\operatorname{id}}

\newcommand{\lsp}{\operatorname{span}}
\newcommand{\clsp}{\overline{\lsp}}
\newcommand{\Per}{\operatorname{Per}}

\newcommand{\PER}{\text{Per}(\Lambda)}

\newcommand{\Prim}{\operatorname{Prim}}

\newcommand{\MT}{\operatorname{MT}}

\title[Primitive ideals of $k$-graph $C^*$-algebras]{The primitive ideals of the
Cuntz-Krieger algebra of a row-finite higher-rank graph with no sources}
\author[Carlsen]{Toke Meier Carlsen}
\address{Toke Carlsen\\ Department of Mathematical Sciences\\ NTNU\\ NO-7491\\ Trondheim\\ Norway}
\email{Toke.Meier.Carlsen@math.ntnu.no}

\author[Kang]{Sooran Kang}
\address{Sooran Kang\\ Department of Mathematics and Statistics\\University of Otago\\PO Box 56\\Dunedin 9054\\New Zealand}
\email{sooran@maths.otago.ac.nz}

\author[Shotwell]{Jacob Shotwell}
\address{Jacob Shotwell\\ TBLR\\ Arizona State University\\ Tempe\\ AZ 85281\\ USA}
\email{shotwell@asu.edu}

\author[Sims]{Aidan Sims}
\address{Aidan Sims\\School of Mathematics and Applied Statistics\\University of Wollongong\\Wollongong 2522\\Australia}
\email{asims@uow.edu.au}

\date{May 28, 2013}

\thanks{This research was supported by the Australian Research Council.}

\begin{document}

\begin{abstract}
We catalogue the primitive ideals of the Cuntz-Krieger algebra of a row-finite
higher-rank graph with no sources. Each maximal tail in the vertex set has an abelian
periodicity group of finite rank at most that of the graph; the primitive ideals in the
Cuntz-Krieger algebra are indexed by pairs consisting of a maximal tail and a character
of its periodicity group. The Cuntz-Krieger algebra is primitive if and only if the whole
vertex set is a maximal tail and the graph is aperiodic.
\end{abstract}

\maketitle

\section{Introduction}

Graph $C^*$-algebras were introduced in the early 1980's by Enomoto and Watatani
\cite{EnomotoWatatani:MJ80} as an alternative description of the Cuntz-Krieger algebras
invented in \cite{CuntzKrieger:IM80}. Enomoto and Watatani considered only finite graphs,
but since the late 1990's substantial work has gone into describing and understanding the
analogous construction for infinite directed graphs in various levels of generality (to
name just a few, \cite{BatesHongEtAl:IJM02, FowlerLacaEtAl:PAMS00, JeongParkEtAl:PJM01,
KumjianPaskEtAl:JFA97, RaeburnSzyma'nski:TAMS04, Tomforde:NYJM01}). A gem in this program
is Hong and Szyma\'nski's description \cite{HongSzyma'nski:JMSJ04} of the primitive ideal
space of the $C^*$-algebra of a directed graph. Hong and Szyma\'nski catalogue the
primitive ideals of $C^*(E)$ in terms of elementary structural features of $E$. They also
describe the closure operation in the hull-kernel topology in terms of this catalogue.

In 1999, Robertson and Steger discovered a class of higher-rank Cuntz-Krieger algebras
arising from $\ZZ^k$ actions on buildings \cite{RobertsonSteger:JRAM99}. Shortly
afterwards, Kumjian and Pask introduced higher-rank graphs and the associated
$C^*$-algebras \cite{KP1} as a simultaneous generalisation of the graph $C^*$-algebras of
\cite{KPR} and Robertson and Steger's higher-rank Cuntz-Krieger algebras. Kumjian and
Pask's $k$-graph algebras have since attracted a fair bit of attention. But their
structure is more subtle and less well understood than that of graph $C^*$-algebras. In
particular, while large portions of the gauge-invariant theory of graph $C^*$-algebras
can be generalised readily to $k$-graphs, higher-rank analogues of other structure
theorems for graph $C^*$-algebras are largely still elusive.

Our main result here is a complete catalogue of the primitive ideals in the $C^*$-algebra
of a row-finite $k$-graph with no sources. Our methods are very different from Hong and
Szyma\'nski's, and require a different set of technical tools. We have been unable to
describe the hull-kernel topology on $\Prim(C^*(\Lambda))$, and we leave this question
open for the present.

\smallskip

We begin in Section~\ref{sec:P-graphs} by introducing $P$-graphs $\Gamma$ and their
$C^*$-algebras $C^*(\Gamma)$, where $P$ is the image of $\NN^k$ under a homomorphism $f$
of $\ZZ^k$. We prove a gauge-invariant uniqueness theorem (Proposition~\ref{prp:Pgraph
giut}) and a Cuntz-Krieger uniqueness theorem (Corollary~\ref{cor:Pgraph CKUT}) for these
$C^*$-algebras. In Section~\ref{sec:pullback algs}, we consider the primitive ideal space
of the pullback $k$-graph arising from a $P$-graph. We first show that if $f$ and $P$ are
as above, then the pullback $f^*\Gamma$ of a $P$-graph $\Gamma$ over $f : \NN^k \to P$
and $d : \Gamma \to P$ is a $k$-graph. We then prove in Theorem~\ref{thm:bijection} that
the irreducible representations of $C^*(f^*\Gamma)$ are in bijection with pairs $(\pi,
\chi)$ where $\pi$ is an irreducible representation of $C^*(\Gamma)$, and $\chi$ is a
character of $\ker f$.

In Section~\ref{sec:max tails}, we study maximal tails in $k$-graphs. Building on an idea
from \cite{DavidsonYang:CJM09}, we show that each maximal tail $T$ has a well-defined
\emph{periodicity group} $\Per(T)$, and contains a large hereditary subset $H$ such that
the subgraph $H\Lambda T$ consisting of paths whose range and source both belong to $H$
is isomorphic to a pullback of an $(\NN^k/\Per(T))$-graph. The algebra $C^*(H\Lambda T)$
can be identified with a corner in $C^*(\Lambda T)$. Combining this with our earlier
results, we describe a bijection $(T, \chi) \mapsto I_{T, \chi}$ from pairs $(T, \chi)$
where $T$ is a maximal tail of $\Lambda$ and $\chi$ is a character $\Per(T)$ to primitive
ideals of $C^*(\Lambda)$. We conclude by showing that $C^*(\Lambda)$ is primitive if and
only if $\Lambda^0$ is a maximal tail and $\Lambda$ is aperiodic.

As usual in our subject, our convention is that in the context of $C^*$-algebras,
``homomorphism" always means ``$^*$-homomorphism," and ``ideal" always means ``closed
two-sided ideal".

\section{\texorpdfstring{$P$}{P}-graphs}\label{sec:P-graphs}

We introduce $P$-graphs over finitely generated cancellative abelian monoids $P$, and an
associated class of $C^*$-algebras. Our treatment is very brief since these objects are
introduced as a technical tool, and the ideas are essentially identical to those
developed in \cite{KP1}. Specialising to $P = \NN^k$ in this section also serves to
introduce the notation used later for $k$-graphs.

\begin{dfn}
Let $P$ be a finitely generated cancellative abelian monoid, which we also regard as a
category with one object. A \emph{$P$-graph} is a countable small category $\Gamma$
equipped with a functor $d : \Gamma \to P$ which has the factorisation property: whenever
$\xi \in \Gamma$ satisfies $d(\xi) = p + q$, there exist unique elements $\eta, \zeta \in
\Gamma$ such that $d(\eta) = p$, $d(\zeta) = q$ and $\xi = \eta\zeta$.
\end{dfn}

Observe that if $P = \NN^k$, then a $P$-graph is precisely a $k$-graph in the sense of
\cite{KP1}.

If $\Gamma$ is a $P$-graph and $p \in P$, then we write $\Gamma^p$ for $d^{-1}(p)$. When
$\xi$ is a morphism of $\Gamma$, we write $s(\xi)$ for the domain of $\xi$ and $r(\xi)$
for the codomain of $\xi.$ If $d(\xi) = 0$, then the factorisation property combined with
the identities $\xi = \id_{r(\xi)}\xi = \xi\id_{s(\xi)}$ imply that $\xi = \id_{r(\xi)} =
\id_{s(\xi)}$. Thus $\Gamma^0 = \{\id_o : o\text{ is an object of }\Gamma\}$, and we can
regard $r,s$ as maps from $\Gamma$ to $\Gamma^0$; we then have $r(\xi)\xi = \xi = \xi
s(\xi)$ for all $\xi \in \Gamma$.

Given $\xi, \eta \in \Gamma$ and a subset $E \subseteq \Gamma$, we define
\begin{gather*}
\xi E = \{\xi\zeta : \zeta \in E, r(\zeta) = s(\xi)\},\quad
    E\eta = \{\zeta\eta : \zeta \in E, s(\zeta) = r(\eta)\},\quad\text{and}\\
    \xi E \eta = \xi E \cap E\eta.
\end{gather*}
For $X, E, Y \subseteq \Gamma$, we write $XEY$ for $\bigcup_{\xi \in X, \eta\in Y} \xi
E\eta$. We say that the $P$-graph $\Gamma$ is \emph{row-finite with no sources} if $0 <
|v\Gamma^p| < \infty$ for all $v \in \Gamma^0$ and $p \in P$.

\begin{example}
Let $P$ be a finitely generated cancellative abelian monoid. Let
\[
    \Omega_P := \{(p,q) \in P \times P : \text{ there exists } r \in P\text{ such that } p+r = q\}.
\]
Since $P$ is cancellative, we may embed it in its Grothendieck group $G$ so that it makes
sense to write $p - q$ for $p,q \in P$. If $(p,q) \in \Omega_P$, then $q - p \in P$.
Define $d : \Omega_P \to P$ by $d(p,q) = q-p$. Define $r,s : \Omega_P \to P$ by $r(p,q) =
p$ and $s(p,q) = q$, and for $(p,q), (q,r) \in \Omega_P$, define $(p,q)(q,r) = (p,r)$;
this $(p,r)$ belongs to $\Omega_P$ because $r = p + ((q-p) + (r-q))$. Then $\Omega_P$ is
a category with objects $P$ and identity morphisms $(p,p)$, and $d$ is a functor from
$\Omega_P$ to $P$. Cancellativity of $P$ implies that $\Omega_P$ is a $P$-graph. We
identify $\Omega_P^0$ with $P$ via $(p,p) \mapsto p$. For $v \in \Omega_P^0$ and $p \in
P$, we then have $v\Omega_P^p = \{(v, v+p)\}$, so $\Omega_P$ is row-finite with no
sources.
\end{example}

A \emph{$P$-graph morphism} is a functor $x : \Gamma \to \Sigma$ between $P$-graphs
$\Gamma$ and $\Sigma$ such that $d_\Sigma(x(\xi)) = d_\Gamma(\xi)$ for all
$\xi \in \Gamma$. For a $P$-graph $\Gamma$, we define
\[
\Gamma^\Omega := \{x : \Omega_P \to \Gamma \mid x\text{ is a $P$-graph morphism}\}.
\]

\begin{lem}\label{lem:construct path}
Let $P$ be a finitely generated cancellative abelian monoid and let $\Gamma$ be a
row-finite $P$-graph with no sources. For each $v \in \Gamma^0$ there exists $x \in
\Gamma^\Omega$ such that $x(0) = v$.
\end{lem}
\begin{proof}
Fix generators $a_1, \dots, a_k$ for $P$. Let $\mathbf{1} = \sum^k_{i=1} a_i$. Let
$v_0:=v$. Recursively for $n \ge 1$ choose $\lambda_n \in v_{n-1}\Lambda^{\mathbf{1}}$
and set $v_n := s(\lambda_n)$. Suppose that $(p,q) \in \Omega_P$. Express $q =
\sum^k_{i=1} q_i a_i$ where each $q_i \in \NN$. Let $n = \max_{i\le k} q_i$. Then
$d(\lambda_1 \dots \lambda_n) = p + (q-p) + \sum^k_{i=1} (n - q_i)a_i$. Two applications
of the factorisation property give $\lambda_1 \dots \lambda_n = \xi \eta \zeta$ where
$d(\xi) = p$ and $d(\eta) = q - p$. Define $x(p,q) := \eta$. The factorisation property
guarantees that there is a well-defined function $x : \Omega_P \to \Gamma$ defined by
this formula, and that $x$ is a $P$-graph morphism. We have $x(0) = v$ by construction.
\end{proof}

We associate to each row-finite $P$-graph $\Gamma$ with no sources a $C^*$-algebra
$C^*(\Gamma)$. As we shall see in section~\ref{sec:pullback algs}, these $C^*$-algebras
are isomorphic to quotients by primitive ideals of crucial building blocks of $k$-graph
$C^*$-algebras.

\begin{dfn}\label{dfn:CKrels}
Let $P$ be a finitely generated cancellative abelian monoid and let $\Gamma$ be a
row-finite $P$-graph with no sources. A \emph{Cuntz-Krieger $\Gamma$-family} in a
$C^*$-algebra $B$ is a collection $\{t_\xi : \xi \in \Gamma\} \subseteq B$ such that
\begin{enumerate}\renewcommand{\theenumi}{CK\arabic{enumi}}
\item $\{t_v : v \in \Gamma^0\}$ is a set of mutually orthogonal projections;
\item $t_\xi t_\eta = t_{\xi\eta}$ whenever $s(\xi) = r(\eta)$;
\item $t_\xi^* t_\xi = t_{s(\xi)}$ for all $\xi \in \Gamma$; and
\item $t_v = \sum_{\xi \in v\Lambda^p} t_\xi t^*_\xi$ for all $v \in \Lambda^0$ and
    $p \in P$.
\end{enumerate}
\end{dfn}

Fix a Cuntz-Krieger $\Gamma$-family $\{t_\xi : \xi \in \Gamma\}$. For $\xi, \eta \in
\Gamma$,
\[
t^*_\xi t_\eta
    = t^*_\xi \sum_{\zeta \in r(\eta)\Gamma^{d(\xi)+d(\eta)}} t_\zeta t^*_\zeta t_\eta.
\]
Fix $\zeta \in r(\eta)\Gamma^{d(\xi)+d(\eta)}$. By the factorisation property, there are
unique factorisations $\zeta = \theta\theta' = \omega\omega'$ with $d(\theta) = d(\xi)$
and $d(\omega) = d(\eta)$. Relations (CK1)~and~(CK3) show that $t^*_\xi t_\zeta = 0$
unless $\theta = \xi$, and $t^*_\omega t_\eta = 0$ unless $\omega = \eta$, and we deduce
that
\[
t^*_\xi t_\eta
    = \sum_{\xi\xi' = \eta\eta' \in r(\eta)\Gamma^{d(\xi)+d(\eta)}} t_{\xi'} t_{\eta'}^*.
\]
Hence $C^*(\{t_\xi : \xi \in \Gamma\}) = \clsp\{t_\xi t^*_\eta : \xi,\eta \in \Gamma\}$.

A standard argument now shows that there is a universal $C^*$-algebra $C^*(\Gamma)$
generated by a Cuntz-Krieger family $\{s_\xi : \xi \in \Gamma\}$. If $P = \NN^k$ then the
relations of Definition~\ref{dfn:CKrels} are those of \cite{KP1}, and so the universal
$C^*$-algebra $C^*(\Gamma)$ coincides with the $C^*$-algebra of $\Gamma$ regarded as a
$k$-graph. To see that the generators of $C^*(\Gamma)$ are all nonzero, we introduce the
groupoid $\Gg_\Gamma$.

For each $\xi \in \Gamma$, define $Z(\xi) \subseteq \Gamma^\Omega$ by
\[
Z(\xi) := \{x \in \Gamma^\Omega : x(0, d(\xi)) = \xi\}.
\]

The $Z(\xi)$ constitute a base of compact open sets for a second-countable locally
compact Hausdorff topology on $\Gamma^\Omega$. For $x \in \Gamma^\Omega$ and $p \in P$,
there is a unique element $\sigma^p(x)$ of $\Gamma^\Omega$ defined by $\sigma^p(x)(q,r) =
x(p+q, p+r)$. Define $\Gg_\Gamma := \{(x, p-q, y) : x,y \in \Omega^P, \sigma^p(x) =
\sigma^q(y)\}$. This is a groupoid with $r(x,g,y) = (x,0,x)$, $s(x, g, y) = (y,0,y)$,
$(x,g,y)(y,h,z) = (x,g+h,z)$ and $(x,g,y)^{-1} = (y, -g, x)$. Arguments like those of
\cite{KP1} show that $\Gg_\Gamma$ is a locally compact Hausdorff \'etale groupoid under
the topology generated by the sets
\[
Z(\xi,\eta) = \{(x, d(\xi) - d(\eta), y) : x \in Z(\xi),
    y \in Z(\eta), \sigma^{d(\xi)}(x) = \sigma^{d(\eta)}(y)\}.
\]
Each $Z(\xi,\eta)$ is compact open in $\Gg_\Gamma$.

We recall from \cite{Renault:groupoidapproachto80} the construction of
$C^*_r(\Gg_\Gamma)$. The space $C_c(\Gg_\Gamma)$ is a $^*$-algebra under the convolution
product $(a * b)(x, g, y) = \sum_{(x,h,z) \in \Gg_\Gamma} a(x, h, z) b(z, h^{-1}g, y)$
and the involution $a^*(x, g, y) = \overline{a(y, g^{-1}, x)}$. For $y \in
\Gamma^\Omega$, let $\Hh_y := \ell^2(\{\alpha \in \Gg_\Gamma : s(\alpha) = y\})$ with
orthonormal basis $\{\delta_\alpha: \alpha\in \Gg_\Gamma, s(\alpha)=y\}$. The regular
representation $\rho_y$ of $C_c(\Gg_\Gamma)$ on $\Hh_y$ is given by $\rho_y(a)\delta_{(x,
g, y)} = \sum_{(z, h, x) \in \Gg_\Gamma} a(z,h,x) \delta_{(z, h+g, y)}$. The reduced
groupoid $C^*$-algebra $C^*_r(\Gg_\Gamma)$ is the closure of the image of $C_c(\Gamma)$
under $\bigoplus_y \rho_y$; equivalently, it is the completion of $C_c(\Gg_\Gamma)$ in
the norm $\|a\| = \sup_y \|\rho_y(a)\|$.

\begin{lem}\label{lem:nonzero}
Let $P$ be a finitely generated cancellative abelian monoid. Let $\Gamma$ be a
row-finite $P$-graph with no sources. For $\xi \in
\Gamma$, let $t_\xi := 1_{Z(\xi, s(\xi))} \in C_c(\Gg_\Gamma) \subseteq
C^*_r(\Gg_\Gamma)$. Then $\{t_\xi : \xi \in \Gamma\}$ is a Cuntz-Krieger $\Gamma$-family,
and there is a homomorphism $\pi_t : C^*(\Gamma) \to C^*_r(\Gg_\Gamma)$ which carries
each $s_\xi$ to $1_{Z(\xi, s(\xi))}$. In particular, each generator of $C^*(\Gamma)$ is
nonzero.
\end{lem}
\begin{proof}
Routine calculations show that the $t_\xi$ form a Cuntz-Krieger $\Gamma$-family. The
existence of the homomorphism $\pi_t$ then follows from the universal property of
$C^*(\Gamma)$. Since the $Z(\xi, s(\xi))$ are all nonempty, the $t_\xi$ are all nonzero
and hence the $s_\xi$ are also nonzero.
\end{proof}

If $G$ is the Grothendieck group of $P$, then for each character $\chi$ of $G$, the
elements $\gamma_\chi(s_\xi) := \chi(d(\xi))s_\xi \in C^*(\Gamma)$ form a Cuntz-Krieger
$\Gamma$ family, and hence induce an endomorphism $\gamma_\chi$ of $C^*(\Gamma)$. Since
$\gamma_{\chi^{-1}}$ is an inverse for $\gamma_\chi$, both are automorphisms. An
$\varepsilon/3$-argument shows that $\gamma$ is a continuous action of $\widehat{G}$ by
automorphisms. Averaging over $\gamma$ gives a faithful conditional expectation $\Phi$
onto the fixed-point algebra $C^*(\Gamma)^\gamma$
\cite[Proposition~3.2]{Raeburn:Graphalgebras05}. Since every spanning element of
$C^*(\Gamma)$ belongs to one of the spectral subspaces of $\gamma$, we have
\begin{equation}\label{eq:cond}
\Phi(s_\xi s^*_\eta) = \delta_{d(\xi), d(\eta)} s_\xi s^*_\eta.
\end{equation}
In particular, we have $C^*(\Gamma)^\gamma = \clsp\{s_\xi s^*_\eta : d(\xi) = d(\eta)\}$.

For a countable set $X$, we write $\mathcal{K}_X$ for the unique nonzero $C^*$-algebra
generated by elements $\{\theta_{x,y}:x,y\in X\}$ such that $\theta_{x,y}^*=\theta_{y,x}$
and $\theta_{x,y}\theta_{w,z}=\delta_{y,w}\theta_{x,z}$.

\begin{lem}\label{lem:direct limit}
Let $P$ be a finitely generated cancellative abelian monoid and let $\Gamma$ be a
row-finite $P$-graph with no sources. Then $C^*(\Gamma)^\gamma$ is an AF algebra. A
homomorphism $\phi : C^*(\Gamma) \to B$ restricts to an injection of $C^*(\Gamma)^\gamma$
if and only if $\phi(s_v) \not= 0$ for all $v$.
\end{lem}
\begin{proof}
Let $a_1, \dots, a_k$ be generators for $P$. For each $n \in \NN$, let $n\cdot\mathbf{1}
:= \sum^k_{i=1} n\cdot a_i$, and let $A_n := \clsp\{s_\xi s^*_\eta : d(\xi) = d(\eta) =
n\cdot\mathbf{1}\}$. The Cuntz-Krieger relations ensure that the $s_\xi s^*_\eta$ are
matrix units. By Lemma~\ref{lem:nonzero} all the $s_\xi$ are nonzero, and hence $s(\xi) =
s(\eta)$ implies $\|s_\xi s^*_\eta\|^2 = \|s_\eta s^*_\xi s_\xi s^*_\eta\| = \|s_\eta
s^*_\eta\| = \|s_\eta\|^2 \not= 0$. It follows that $s_\xi s^*_\eta \mapsto \Theta_{\xi,
\eta}$ determines an isomorphism $A_n \cong \bigoplus_{v \in \Gamma^0}
\Kk_{\Gamma^{n\cdot\mathbf{1}} v}$. If $m \le n$, then for $v \in \Gamma^0$ and $\xi,\eta
\in \Gamma^{m\cdot\mathbf{1}}v$, we have
\[
s_\xi s^*_\eta = \sum_{\zeta \in v\Gamma^{(n-m)\cdot\mathbf{1}}} s_{\xi\zeta} s^*_{\eta\zeta} \in A_n,
\]
and hence $m \le n$ implies $A_m \subseteq A_n$. Hence $\overline{\bigcup_n A_n}$ is an
AF subalgebra of $C^*(\Gamma)^\gamma$. Suppose that $\xi,\eta \in \Gamma$ satisfy
$s(\xi) = s(\eta) = v$ and $d(\xi) = d(\eta)$. Write $d(\xi) = \sum^k_{i=1} p_i a_i$. Let
$n := \max_i p_i$ and let $q := \sum^k_{i=1} (n - p_i)a_i$. Then $n\cdot\mathbf{1} = d(\xi) + q$ and
 $s_\xi s^*_\eta = \sum_{\zeta \in v\Gamma^q} s_{\xi\zeta} s^*_{\eta\zeta} \in A_n$.
Hence $C^*(\Gamma)^\gamma = \overline{\bigcup_n A_n}$ is AF.

Now suppose that $\phi : C^*(\Gamma) \to B$ is a homomorphism. Since each $s_v$ is
nonzero and belongs to $C^*(\Gamma)^\gamma$, if $\phi$ is injective on
$C^*(\Gamma)^\gamma$, then each $\phi(s_v) \not= 0$. Conversely suppose that each
$\phi(s_v) \not= 0$. Fix $n \in \NN$. For $\xi, \eta \in \Gamma^{n\cdot\mathbf{1}}$, the
reasoning of the first paragraph of this proof shows that $\phi(s_\xi s^*_\eta) \not= 0$.
Since each $\Kk_{\Gamma^{n\cdot\mathbf{1}} v}$ is simple, it follows that $\phi$ is
injective and hence isometric on $A_n$. Since the $A_n$ are nested, it follows that
$\phi$ is isometric on $\bigcup_n A_n$, and therefore on $\overline{\bigcup_n A_n} =
C^*(\Gamma)^\gamma$ as well.
\end{proof}

\begin{prop}\label{prp:Pgraph giut}
Let $P$ be a finitely generated cancellative abelian monoid, and let $G$ be its
Grothendieck group. Let $\Gamma$ be a row-finite $P$-graph with no sources. Suppose that
$\{t_\xi : \xi \in \Gamma\}$ is a Cuntz-Krieger $\Gamma$-family in a $C^*$-algebra $B$
and there is an action $\beta$ of $\widehat{G}$ on $B$ such that $\beta_\chi(t_\xi) =
\chi(d(\xi))t_\xi$ for all $\xi \in \Gamma$. Then the induced homomorphism $\pi_t :
C^*(\Gamma) \to B$ is injective if and only if $t_v \not= 0$ for all $v \in \Gamma^0$.
The homomorphism $\pi_t$ of Lemma~\ref{lem:nonzero} is an isomorphism from $C^*(\Gamma)$
to $C^*_r(\Gg_\Gamma)$.
\end{prop}
\begin{proof}
First observe that if some $t_v = 0$ then Lemma~\ref{lem:nonzero} implies that $\pi_t$ is
not injective. Now suppose that each $t_v \not= 0$. By Lemma~\ref{lem:direct limit}, the
homomorphism $\pi_t$ is injective on $C^*(\Gamma)^\gamma$. Averaging over $\beta$ gives a
conditional expectation $\Psi : \pi_t(C^*(\Gamma)) \to \pi_t(C^*(\Gamma)^\gamma)$ such
that $\Psi \circ \pi_t = \pi_t \circ \Phi$, where $\Phi$ is the conditional expectation
of \eqref{eq:cond}. Now the following standard argument shows that $\pi_t$ is faithful:
\[
\pi_t(a) = 0
    \implies \Psi(\pi_t(a^*a)) = 0
    \implies \pi_t(\Phi(a^*a)) = 0
    \implies \Phi(a^*a) = 0
    \implies a = 0.
\]

Now let $\{t_\xi : \xi \in \Gamma\}$ and $\pi_t : C^*(\Gamma) \to C^*_r(\Gg_\Gamma)$ be
as in Lemma~\ref{lem:nonzero}. Define $c : \Gg_\Gamma \to G$ by $c(x,g,y) = g$. For $y
\in \Gamma^\Omega$, let $\rho_y$ be the regular representation discussed
prior to Lemma~\ref{lem:nonzero}. For $\alpha,\beta\in \Gg_\Gamma$ with
$s(\alpha)=s(\beta)=y$, let $\Theta_{\alpha,\beta} \in \Bb(\Hh_y)$ be the rank-one
operator from $\mathbb{C}\delta_\beta$
to $\mathbb{C}\delta_\alpha$. There is a strongly continuous action $\beta^y$ of
$\widehat{G}$ on $\Bb(\Hh_y)$ such that $\beta^y_\chi(\Theta_{\alpha,\beta}) =
\chi(c(\alpha) - c(\beta)) \Theta_{\alpha, \beta}$. In particular,
$\beta^y_\chi(\rho_y(t_\xi)) = \chi(d(\xi))\rho_y(t_\xi)$ for all $\xi \in
\Gamma$ and $\chi\in\widehat{G}$. Thus $\beta = \bigoplus_y \beta^y$ is a strongly
continuous action of $\widehat{G}$ on $C^*_r(\Gg_\Gamma)$ such that $\beta_\chi(t_\xi) =
\chi(d(\xi))t_\xi$ for all $\xi \in \Gamma$. So $\pi_t$ is injective.
Lemma~\ref{lem:nonzero} implies that if $s(\xi) = s(\eta)$ in $\Gamma$, then
$\pi_t(s_\xi s^*_\eta) = 1_{Z(\xi, s(\xi))} 1_{s(\eta), \eta} = 1_{Z(\xi, \eta)}$. The
$C^*$-norm on $C^*_r(\Gg_\Gamma)$ coincides with the supremum norm on each
$C_c(Z(\xi, \eta))$, so the Stone-Weierstrass theorem implies that the range of $\pi_t$
contains $C_c(\Gg_\Gamma)$, and hence $\pi_t$ is surjective.
\end{proof}

Let $P$ be a finitely generated cancellative abelian monoid and let $\Gamma$ be a
$P$-graph. We say that $\Gamma$ is \emph{aperiodic} if for every $v \in \Gamma^0$ there
exists $x \in \Gamma^\Omega$ such that $p \not= q \in P$ implies $\sigma^p(x) \not=
\sigma^q(x)$.

\begin{cor}\label{cor:Pgraph CKUT}
Let $P$ be a finitely generated abelian monoid and let $\Gamma$ be a row-finite $P$-graph with no sources. Suppose
that $\Gamma$ is aperiodic. If $\{t_\xi : \xi \in \Gamma\}$ is a Cuntz-Krieger
$\Gamma$-family in a $C^*$-algebra $B$, then the induced homomorphism $\pi_t :
C^*(\Gamma) \to B$ is injective if and only if $t_v \not= 0$ for every $v \in \Gamma^0$.
\end{cor}
\begin{proof}
The only if is clear because Lemma~\ref{lem:nonzero} implies that each $s_v \not= 0$.
Suppose that each $t_v \not= 0$. Then each $t_\lambda t^*_\lambda \not= 0$. By
Proposition~\ref{prp:Pgraph giut}, we can identify $C^*(\Gamma)$ with
$C^*_r(\Gg_\Gamma)$, and regard $\pi_t$ as a homomorphism of $C^*_r(\Gg_\Gamma)$ which
carries each $1_{Z(\xi, s(\xi))}$ to $t_\xi$. The isomorphism $C^*(\Gamma) \cong
C^*_r(\Gg_\Gamma)$ identifies $C_0(\Gg_\Gamma^{(0)})$ with a subalgebra of
$C^*(\Gamma)^\gamma$, so Lemma~\ref{lem:direct limit} implies that $\pi_t$ is injective
on $C_0(\Gg_\Gamma^{(0)})$.

Fix a basic open set $Z(\xi)$ in $\Gg_\Gamma^{(0)} = \Gamma^\Omega$. By hypothesis, there
exists $y \in \Gamma^\omega$ such that $y(0) = s(\xi)$ and $\sigma^p(y) \not=
\sigma^q(y)$ whenever $p \not= q$. The factorisation property implies that there is a
unique element $\xi y$ of $Z(\xi)$ such that $\sigma^{d(\xi)}(\xi y) = y$. We have
\[
\sigma^p(\xi y) = \sigma^q(\xi y)
    \implies \sigma^{d(\xi)}(\sigma^p(\xi y)) = \sigma^{d(\xi)}(\sigma^q(\xi y))
    \implies \sigma^p(y) = \sigma^q(y)
    \implies p = q.
\]
Thus $\Gg_\Lambda$ has trivial isotropy at $\xi y$. Thus $\Gg_\Gamma$ is topologically
principal. It now follows from \cite[Theorem~4.4]{Exel:PAMS10} that every nontrivial
ideal of $C^*_r(\Gg_\Gamma)$ has nontrivial intersection with $C_0(\Gg_\Gamma^{(0)})$. In
particular, that $\pi_t|_{C_0(\Gg_\Gamma^{(0)})}$ is injective shows that $\pi_t$ is
injective.
\end{proof}


\section{The primitive ideal space of the \texorpdfstring{$C^*$}{C*}-algebra of a pullback}\label{sec:pullback algs}

In this section we consider pullback $k$-graphs of the form $f^*\Gamma$ where $f : \ZZ^k
\to G$ is a group homomorphism and $\Gamma$ is a $P$-graph for $P = f(\NN^k)$.
Proposition~\ref{central  unitaries} and Lemma~\ref{quotient iso} combine to show that,
putting $H := \ker(f)$, the $C^*$-algebra $C^*(f^*\Gamma)$ is a $C(\widehat{H})$-algebra
with fibres identical to $C^*(\Gamma)$. We use this to give a complete listing of the
irreducible representations of $C^*(f^*\Gamma)$ in terms of the irreducible
representations of $C^*(\Gamma)$ and characters of $\ker f$. We begin by introducing
pullbacks of $P$-graphs.

\begin{dfn}[{cf. \cite[Definition~1.9]{KP1}}]
Let $P$ and $Q$ be finitely generated cancellative abelian monoids, and let
$f:P\rightarrow Q$ be a monoid morphism. If $(\Gamma,d)$ is a $Q$-graph, we define the
$P$-graph $f^*\Gamma$ as follows: $f^*\Gamma=\{(\lambda,n):d(\lambda)=f(n)\}$ with
$d(\lambda,n)=n$, $s(\lambda,n)=s(\lambda)$ and $r(\lambda,n)=r(\lambda)$. Composition is
given by $(\mu,m)(\nu,n) = (\mu\nu,m+n)$.
\end{dfn}

For $g,h \in \ZZ^k$, we write $g \vee h$ for the coordinatewise maximum of $g$ and $h$
and $g \wedge h$ for the coordinatewise minimum. Given $h \in \ZZ^k$ we define $h_+ := h
\vee 0$ and $h_- := -(h \wedge 0)$. We then have
$h = h_+ - h_-$ with $h_+ \wedge h_- = 0$.

\begin{lem}\label{prop1}
Let $H$ be a subgroup of $\ZZ^k$, let $G = \ZZ^k/H$ and let $f : \ZZ^k \to G$ be the
quotient map. Let $P = f(\NN^k) \subseteq G$. Suppose that $\Gamma$ is a row-finite
$P$-graph with no sources. Then $f^*\Gamma$ is a row-finite $k$-graph with no sources.
Suppose that $h\in H$ and $\lambda\in \Gamma$ satisfy $d(\lambda)=f(h_+)$. Then
$(\lambda,h_+)$ and $(\lambda,h_-)$ both belong to $f^*\Gamma$.
\end{lem}

\begin{proof}
Clearly $f^*\Gamma$ is a countable category and $(\lambda,n)\mapsto n$ is a functor. We
check the factorisation property; if $(\lambda,m+n)\in f^*\Gamma$, then
$d(\lambda)=f(m)+f(n)$. So $\lambda$ factorises uniquely as $\lambda=\mu\nu$ with
$d(\mu)=f(m)$, $d(\nu)=f(n)$, and then $(\lambda,m+n)=(\mu,m)(\nu,n)$ is the unique
factorisation with $d(\mu,m)=m$ and $d(\nu,n)=n$.

For the second assertion, observe that $f(h_+)=f(h_-+h)=f(h_-)+f(h)=f(h_-)$ since $h\in
H=\ker(f)$.
\end{proof}

In the following proposition, $\{U_h : h \in H\}$ denotes the canonical collection of
unitary generators of the group $C^*$-algebra $C^*(H)$. We write $\mathcal{ZM}(A)$ for
the centre of the multiplier algebra of a $C^*$-algebra $A$.

\begin{prop}\label{central  unitaries}
Let $G,f,P,H$ be as in Lemma~\ref{prop1}. Let $\Gamma$ be a row-finite $P$-graph with no
sources. For each $h\in H$ and $v\in\Gamma^0$, let $u_{v,h}:=\sum_{\lambda\in
v\Gamma^{f(h_+)}}s_{(\lambda,h_+)}s^*_{(\lambda,h_-)}$. Then $\sum_{v\in\Gamma} u_{v,h}$
converges strictly to a central unitary multiplier
$V_h:=\sum_{d(\lambda)=f(h_+)}s_{(\lambda,h_+)}s^*_{(\lambda,h_-)}$ of $C^*(f^*\Gamma)$.
Moreover, there is an injective homomorphism $\rho: C^*(H)\to
\mathcal{ZM}(C^*(f^*\Gamma))$ such that $\rho(U_h)=V_h$ for all $h\in H$.
\end{prop}

\begin{proof}
 For $v \in \Gamma^0$ and $h \in H$, we claim that $u_{v,h}$ is a partial
isometry whose initial and final projections are both equal to $s_{(v,0)}$. Indeed
\[\begin{split}
u_{v,h}^* u_{v,h}
        &= \sum_{\lambda,\mu \in v\Gamma^{f(h_+)}} s_{(\lambda,h_-)} s^*_{(\lambda, h_+)} s_{(\mu,h_+)} s^*_{(\mu, h_-)}\\
        &= \sum_{\lambda,\mu \in v\Gamma^{f(h_+)}} \delta_{\lambda,\mu} s_{(\lambda,h_-)} s^*_{(\mu, h_-)}
        = \sum_{\eta \in (v,0)(f^*\Gamma)^{h_-}} s_\eta s^*_\eta
        = s_{(v,0)},
\end{split}\]
and a similar calculation shows that $u_{v,h}u^*_{v,h} = s_{(v,0)}$ also.

By, for example, the argument of \cite[Lemma~2.10]{Raeburn:Graphalgebras05}, for each $a
\in C^*(f^*\Gamma)$, we have $\sum_{v \in F} s_{(v,0)} a \to a$ as $F$ increases over
finite subsets of $\Gamma^0$. So for $a \in C^*(f^*\Gamma)$ and $\varepsilon > 0$, there
exists a finite $F \subseteq \Gamma^0$ such that $\|a - \sum_{v \in F} s_v a\| <
\varepsilon$. Thus, for $K \subseteq L \subseteq \Gamma^0 \setminus F$ and $h\in H$, we
have
\begin{align*}
\Big\|\sum_{v \in L} u_{v,h} a - \sum_{w \in K} u_{w,h} a\Big\|
    &= \Big\|\sum_{v \in L \setminus K} u_{v,h} a\Big\| \\
    &= \Big\|\sum_{v \in L \setminus K} u_{v,h}\Big(a - \sum_{w \in F} s_{(w,0)} a\Big)\Big\| \quad\text{since $L\subseteq \Gamma^0\setminus F$}\\
    &\le \varepsilon\Big\|\sum_{v \in L \setminus K} u_{v,h}\Big\| \\
    &= \varepsilon
\end{align*}
since the $u_{v,h}$ are partial isometries whose initial projections are mutually
orthogonal and whose final projections are also mutually orthogonal. Hence the net
$\big\{\sum_{v \in F} u_{v,h} a : F \subseteq \Gamma^0\text{ is finite}\}$ is Cauchy, so
converges. Thus, the series
\[
    \sum_{v \in \Lambda^0} u_{v,h} = \sum_{d(\lambda)=f(h_+)}s_{(\lambda,h_+)}s^*_{(\lambda,h_-)}
\]
converges strictly to a multiplier $V_h$ of $C^*(f^*\Gamma)$.

Since $(-h)_+ = h_-$ and $(-h)_- = h_+$ for all $h \in \ZZ^k$, we have $u_{v,h}^* =
u_{v,-h}$ for all $v, h$. Hence $V^*_h = V_{-h}$ for all $h \in H$. The $V_h$ are unitary
because
\[
    (V_h^* V_h a) = \lim_{F,K} \sum_{v \in F, w \in K} u_{v,h}^* u_{v,h} a
            = \lim_{F} \sum_{v \in F} s_{(v,0)} a
            = a,
\]
so $V_h^* V_h = 1_{\Mm(C^*(f^*\Gamma))}$, and then $V_h V^*_h = V^*_{-h} V_{-h} =
1_{\Mm(C^*(f^*\Gamma))}$ also.

To see that the $V_h$ are central, observe that
\[
V_h s_{(\lambda,m)} = \lim_{F} \sum_{v \in F} u_{v,h} s_{(\lambda,m)}
    = \lim_F \sum_{v \in F} \sum_{\mu \in v\Gamma^{f(h_+)}} s_{(\mu, h_+)} s^*_{(\mu,h_-)} s_{(\lambda,m)}.
\]
Applying the Cuntz-Krieger relation, we obtain
\[
V_h s_{(\lambda,m)}
    = \lim_F \sum_{v \in F} \sum_{\mu \in v\Gamma^{f(h_+)}}
        \sum_{\substack{\alpha \in s(\mu)\Gamma^{f(m)}\\ \beta \in s(\lambda)\Gamma^{f(h_-)}}}
        s_{(\mu\alpha, h_+ + m)} s^*_{(\mu\alpha, h_- + m)} s_{(\lambda\beta,h_- + m)}s^*_{(\beta,h_-)}.
\]
The only nonzero terms are those where $\mu\alpha = \lambda\beta$. Since $\Gamma$ has no sources,
 for each $\beta \in s(\lambda)\Gamma^{f(h_-)}$ there is a unique
$\mu \in r(\lambda)\Gamma^{f(h_-)}$ and a unique $\alpha \in s(\mu)\Gamma^{f(m)}$ such that
$\mu\alpha = \lambda\beta$. So the final sum above collapses to give
\[
V_h s_{(\lambda,m)}
    = \sum_{\beta \in s(\lambda)\Gamma^{f(h_-)}} s_{(\lambda\beta, h_+ + m)} s^*_{(\beta, h_-)}.
\]
On the other hand,
\[\begin{split}
s_{(\lambda,m)} V_h
    &=s_{(\lambda,m)} \lim_F \sum_{v\in F} \sum_{\beta \in v\Gamma^{f(h_+)}} s_{(\beta, h_+)} s^*_{(\beta, h_-)}\\
    &= \sum_{\beta \in s(\lambda)\Gamma^{f(h_+)}} s_{(\lambda,m)}s_{(\beta, h_+)} s^*_{(\beta, h_-)}
    = \sum_{\beta \in s(\lambda)\Gamma^{f(h_+)}} s_{(\lambda\beta, h_+ + m)} s^*_{(\beta, h_-)}
\end{split}\]
as required.

The universal property of $C^*(H)$ implies that there is a homomorphism $\rho:C^*(H)\to
\mathcal{ZM}(C^*(f^*\Gamma))$ such that $\rho(U_h)=V_h$ for all $h\in H$. To see that
$\rho$ is injective, define an action $\beta$ of $\widehat{H}$ on $C^*(f^*\Gamma)$ by
$\beta_{\chi}(s_{(\lambda,n)}) = \chi(n)s_{(\lambda,n)}$. Then $\beta$ extends to an
action of $\widehat{H}$ on $\Mm C^*(f^*\Gamma)$ such that $\beta_{\chi}(V_h)=\chi(h)
V_h$. So $\rho$ is nonzero and equivariant for $\beta$ and the dual action of
$\widehat{H}$ on $C^*(H)$. Therefore $\rho$ is injective.
\end{proof}

\begin{lem}\label{quotient iso}
Let $G,f,P,H$ be as in Lemma~\ref{prop1}. Let $\Gamma$ be a row-finite $P$-graph with no
sources. For $z \in \TT^k$, let $I_z$ be the ideal of $C^*(f^*\Gamma)$ generated by
$\{z^{-n_1}s_{(\lambda,n_1)} -  z^{-n_2}s_{(\lambda,n_2)}: \lambda\in\Gamma,\ n_1,n_2\in
\NN^k,\ f(n_1)=f(n_2)=d(\lambda)\}$. Then there is an isomorphism $\psi_z :
C^*(f^*\Gamma)/I_z \to C^*(\Gamma)$ such that $\psi_z(s_{(\lambda, n)} + I_z)= z^{n}
s_\lambda$ for all $\lambda \in \Gamma$.
\end{lem}

\begin{proof}
The set $\{z^n s_\lambda : (\lambda, n) \in f^*\Gamma\}$ is a Cuntz-Krieger
$f^*\Gamma$-family. Hence there is a homomorphism $\psi: C^*(f^*\Gamma) \to C^*(\Gamma)$
carrying each $s_{(\lambda,n)}$ to $z^n s_\lambda$. Each generator of  $I_z$ belongs to
$\ker \psi$, and hence $\psi$ descends to a homomorphism $\psi_z : C^*(f^*\Gamma)/I_z \to
C^*(\Gamma)$ satisfying $\psi_z(s_{(\lambda,n)} + I_z) = z^{n} s_\lambda$.

To see that this $\psi_z$ is an isomorphism, we show that it has an inverse.
Fix $\lambda \in \Gamma$ and $n_1, n_2 \in \NN^k$ such that $f(n_1) = f(n_2) = d(\lambda)$.
Then $z^{-n_1}s_{(\lambda, n_1)} - z^{-n_2}s_{(\lambda, n_2)} \in I_z$.
So we may define a collection $\{t_\lambda : \lambda \in \Gamma\} \subseteq C^*(f^*\Gamma)/I_z$
by  $t_\lambda = z^{-n}s_{(\lambda,n)} + I_z$ for any $n \in f^{-1}(d(\lambda))$; in particular $t_v =
s_{(v,0)} +I_z$ for all $v \in \Gamma^0$. We show that the $t_\lambda$ form a Cuntz-Krieger
$\Gamma$-family. The $t_v$ are mutually orthogonal projections because the $s_{(v,0)}$ are.
Suppose that $s(\mu) = r(\nu)$, and fix $m,n$ such that $f(m) = d(\mu)$ and $f(n) = d(\nu)$. Then
$f(m+n) = d(\mu\nu)$, and so
\[
t_\mu t_\nu = z^{-m}s_{(\mu,m)} z^{-n}s_{(\nu,n)} + I_z = z^{-m-n}s_{(\mu\nu,m+n)} + I_z= t_{\mu\nu}.
\]
Moreover, $t^*_\mu t_\mu = s^*_{(\mu,m)} s_{(\mu,m)}+ I_z = s_{(s(\mu), 0)} + I_z =
t_{s(\mu)}$. Fix $v \in \Gamma^0$, $p \in P$ and $m \in \ZZ^k$ with $f(m) =p$. Then
\[
t_{v} = s_{(v,0)} + I_z
    = \sum_{\alpha \in (v,0)(f^*\Gamma)^m} s_\alpha s^*_\alpha + I_z
    = \sum_{\lambda \in v\Gamma^p} s_{(\lambda,m)} s^*_{(\lambda,m)} + I_z
    = \sum_{\lambda \in v\Gamma^p} t_\lambda t^*_\lambda.
\]
So $\{t_\lambda : \lambda \in \Gamma\}$ is a Cuntz-Krieger $\Gamma$-family as claimed.
Hence there is a homomorphism $C^*(\Gamma) \to C^*(f^*\Gamma)/I_z$ satisfying $s_\lambda
\mapsto z^{-n}s_{(\lambda,n)} + I_z$ for any $n \in f^{-1}(d(\lambda))$; this
homomorphism is an inverse for $\psi_z$.
\end{proof}

In the following theorem, given a $C^*$-algebra $A$, we write $\operatorname{Irr}(A)$ for
the collection of all irreducible representations of $A$.

\begin{thm}\label{thm:bijection}
Let $G$ be a finitely generated abelian group, and let $f : \ZZ^k \to G$ be a
homomorphism. Let $P = f(\NN^k) \subseteq G$ and let $H = \ker(f) \subseteq \ZZ^k$. Let
$\Gamma$ be a row-finite $P$-graph with no sources. For $z\in \TT^k$, let $q_z:
C^*(f^*\Gamma) \to C^*(f^*\Gamma)/ I_z$ be the quotient map, and let $\psi_z :
C^*(f^*\Gamma)/I_z \to C^*(\Gamma)$ be the isomorphism of Lemma~\ref{quotient iso}. Let
$\pi$ be an irreducible representation of $C^*(\Gamma)$. Then $\pi\circ \psi_z\circ q_z$
is an irreducible representation of $C^*(f^*\Gamma)$. Fix a map $\gamma\mapsto z_\gamma$
from $\widehat{H}$ to $\TT^k$ such that $z^h_\gamma=\gamma(h)$ for all $\gamma\in
\widehat{H}$ and $h\in H$. Then $(\gamma,\pi)\mapsto \pi\circ \psi_{z_\gamma}\circ
q_{z_\gamma}$ is a bijection of $\widehat{H} \times \operatorname{Irr}(C^*(\Gamma))$ onto
$\operatorname{Irr}(C^*(f^* \Gamma))$.
\end{thm}

\begin{rmk}\label{rmk:unique character}
Theorem~\ref{thm:bijection} together with the definition of $I_z$ implies that if $\pi$
is an irreducible representation of $C^*(f^*\Gamma)$, then there is a unique character
$\gamma$ of $\widehat{H}$ such that $s_{(\lambda, n_1)} - \gamma(n_1 - n_2)s_{(\lambda,
n_2)} \in \ker(\pi)$ for all $\lambda \in \Gamma$ and $n_1, n_2 \in f^{-1}(d(\lambda))$.
\end{rmk}

We collect some more technical lemmas before proving Theorem \ref{thm:bijection}.

\begin{lem}\label{ideal comparison}
Let $G,f,P,H$ be as in Lemma~\ref{prop1}. Let $\Gamma$ be a row-finite $P$-graph with no sources.
For $z \in \TT^k$, let $I_z$ be the ideal of $C^*(f^*\Gamma)$ generated by
$\{z^{-n_1}s_{(\lambda,n_1)} -  z^{-n_2}s_{(\lambda,n_2)}: \lambda\in\Gamma,\ n_1,n_2\in \NN^k,\ f(n_1)=f(n_2)=d(\lambda)\}$.
For $w,z \in \TT^k$ we have $I_w = I_z$ if and only if $w^h = z^h$ for all $h \in H$.
\end{lem}

\begin{proof}
If $w^h = z^h$ for all $h \in H$, then $I_w = I_z$ because they have the same generators.
Now suppose that $I_w = I_z$. Fix $h \in H$. For $v \in \Gamma^0$ we have $w^h s_{(v,0)}
- u_{v,h} \in I_w = I_z \owns z^h  s_{(v,0)} - u_{v,h}$. Subtracting, we obtain $(w^h -
z^h) s_{(v,0)} \in I_z$. The isomorphism $\psi_z$ of Lemma~\ref{quotient iso} carries
each $s_{(v,0)} + I_z$ to $s_v$ which is nonzero by Proposition~\ref{prp:Pgraph giut}.
Hence $s_{(v,0)} \not\in I_z$, which forces $w^h - z^h = 0$. Hence $w^h = z^h$ for all $h
\in H$.
\end{proof}

\begin{lem}\label{representation}
Let $G,f,P,H$ be as in Lemma~\ref{prop1}. Let $\Gamma$ be a row-finite $P$-graph with no sources.
Let $\varphi$ be an irreducible representation of $C^*(f^*\Gamma)$ on a Hilbert space $\Hh$.
Then $\varphi$ has a unique strict extension to a representation
$\widetilde{\varphi}$ of $\Mm C^*(f^*\Gamma)$ and there is a unique character
$\gamma_\varphi$ of $H$ such that $\widetilde{\varphi}(V_h) = \gamma_\varphi(h)1_\Hh$ for
all $h \in H$. For any $z \in \TT^k$ such that $z^h = \gamma_{\varphi}(h)$ for
 all $h \in H$, we have $I_z \subseteq \ker(\varphi)$. There is an irreducible
 representation $\widehat{\varphi}$ of $C^*(f^*\Gamma)/I_z$ such that $\widehat{\varphi}(a
  + I_z)= \varphi(a)$ for all $a$, and $\widehat{\varphi} \circ \psi_z^{-1}$ is an
  irreducible representation of $C^*(\Gamma)$.
\end{lem}

\begin{proof}
The irreducible representation $\varphi: C^*(f^*\Gamma)\to B(\Hh)$ is extendible since it is non-degenerate.
Since $\varphi$ is irreducible, $\widetilde{\varphi}$ is also.
Since the $V_h$ are central in $\Mm C^*(f^*\Gamma)$, the $\widetilde{\varphi}(V_h)$ are central in $\widetilde{\varphi}(\Mm C^*(f^*\Gamma))$.
Therefore the $\widetilde{\varphi}(V_h)$ are all scalar multiples of identity operator $1_{\Hh}$ by, for example, \cite[Theorem~4.1.12]{Murphy}.
That is, for each $h\in H$, there exists $\gamma_\varphi(h)\in \TT$ such that $\widetilde{\varphi}(V_h)=\gamma_\varphi(h)1_{\Hh}$.
Since $\widetilde{\varphi}$ is multiplicative, $\gamma_\varphi:H\to \TT$ is a homomorphism, so $\gamma_\varphi\in \widehat{H}$.
The uniqueness of $\gamma_\varphi$ is obvious.

Fix $z\in \TT^k$ such that $z^h=\gamma_\phi(h)$ for all $h\in H$. For $h\in H$, we have
\[
\varphi(z^h s_{(v,0)}-u_{v,h})=\varphi((z^h V_0-V_h)s_{(v,0)})=\widetilde{\varphi}(z^h V_0-V_h)\varphi(s_{(v,0)})=0.
\]
So $\varphi$ descends to a representation $\widehat{\varphi}$ of $C^*(f^*\Gamma)/I_z$,
which is irreducible because $\varphi$ is.
\end{proof}

\begin{proof}[Proof of Theorem \ref{thm:bijection}]
Since the homomorphism $\psi_z :C^*(f^*\Gamma)/I_z \to C^*(\Gamma)$ of
Lemma~\ref{quotient iso} is an isomorphism, the composition $\psi_z \circ q_z$ is
surjective, and hence $\pi\circ \psi_z\circ q_z$ is an irreducible representation of
$C^*(f^*\Gamma)$.

To see that $(\gamma,\pi)\mapsto \pi\circ \psi_{z_\gamma}\circ q_{z_\gamma}$ is
surjective, fix an irreducible representation $\varphi$ of $C^*(f^*\Gamma)$. Let
$\gamma_{\varphi}$ be the corresponding character of $H$ given in
Lemma~\ref{representation}, and let $\psi_{z_{\gamma_\varphi}}:
C^*(f^*\Gamma)/I_{z_\varphi}\to C^*(\Gamma)$ be the isomorphism of Lemma~\ref{quotient
iso}. By Lemma~\ref{representation}, there is an irreducible representation
$\widehat{\varphi}$ of $C^*(f^*\Gamma)/I_{z_{\gamma_\varphi}}$ such that
$\widehat{\varphi}\circ q_{z_{\gamma_\varphi}}=\varphi$, and then $\widehat{\varphi}\circ
\psi^{-1}_{z_{\gamma_\varphi}}$ is an irreducible representation of $C^*(\Gamma)$. Now
\[
(\widehat{\varphi} \circ \psi^{-1}_{z_{\gamma_\varphi}}) \circ \psi_{z_{\gamma_\varphi}} \circ q_{z_{\gamma_\varphi}}
    = \widehat{\varphi} \circ q_{z_{\gamma_\varphi}}
    =\varphi,
\]
so $(\gamma,\pi)\mapsto \pi\circ \psi_{z_\gamma}\circ q_{z_\gamma}$ is surjective.

To see that it is injective, fix $\pi$ and $\gamma$. We claim that for $z\in\TT^k$, we have $I_z\subseteq \ker(\pi\circ\psi_{z_\gamma}\circ q_{z_\gamma})$ if and only if $z^h=\gamma(h)$ for all $h\in H$. First suppose that $I_z\subseteq \ker(\pi\circ\psi_{z_\gamma}\circ q_{z_\gamma})$. Fix $h\in H$, and let $\theta=\pi\circ\psi_{z_\gamma}\circ q_{z_\gamma}$. Then $\theta(u_{v,h})=z^h_\gamma \pi(s_v)$. Hence
\[
0=\theta(z^h s_{(v,0)}-u_{v,h})=(z^h-\gamma(h))\pi(s_v).
\]
Since $\pi$ is nonzero, there exists $v\in\Gamma^0$ such that $\pi(s_v)\ne 0$ forcing $z^h=\gamma(h)$ for all $h\in H$.
Now suppose that $z^h=\gamma(h)$ for all $h\in H$. Then every generator of $I_z$ belongs to $\ker \theta$, giving $I_z\subseteq\ker\theta$.

Fix  $(\gamma_1, \pi_1), (\gamma_2, \pi_2) \in \widehat{H} \times
\operatorname{Irr}(C^*(\Gamma))$, and write $\theta_i := \pi_i \circ \psi_{z_{\gamma_i}}
\circ q_{z_{\gamma_i}}$ for $i=1,2$. Suppose that $\theta_1 = \theta_2$. Then $I_z
\subseteq \ker(\theta_1)= \ker(\theta_2)$  for all $z \in \TT^k$. Thus $\gamma_1(h) =
\gamma_2(h)$ for all $h \in H$ by the preceding paragraph; that is $\gamma_1 = \gamma_2$.
This forces $\psi_{z_{\gamma_1}} \circ q_{z_{\gamma_1}} = \psi_{z_{\gamma_2}} \circ
q_{z_{\gamma_2}}$. Since these are surjections onto $C^*(\Gamma)$, the equality $\theta_1
= \theta_2$ forces $\pi_1 = \pi_2$.
\end{proof}

\section{Maximal tails, periodicity, and pullbacks}\label{sec:max tails}

Theorem 3.12 of \cite{KaP1} implies that the maximal tails in a strongly aperiodic $k$-graph
index the primitive gauge-invariant ideals in its $C^*$-algebra. These maximal tails are
also a key ingredient in our catalogue of the primitive ideals of an arbitrary $k$-graph
$C^*$-algebra. In this section we show that every maximal tail $T$ in a $k$-graph contains a
hereditary subset $H$ for which the subgraph $H\Lambda T$ is isomorphic to a pullback of
a $P$-graph as described in Section~\ref{sec:P-graphs}. This implies that for gauge-invariant ideals associated to maximal tails, the quotient
contains a corner whose primitive-ideal space is described by
Theorem~\ref{thm:bijection}. This in turn is the key to our main theorem in
Section~\ref{sec:all prim ideals}.

We recall the definition of a maximal tail from \cite{KaP1}.

\begin{dfn}[{\cite[Definition~3.10]{KaP1}}]
Let $\Lambda$ be a row-finite $k$-graph with no sources. A nonempty subset $T$ of
$\Lambda^0$ is called a \textit{maximal tail} if
\begin{enumerate}\renewcommand{\theenumi}{\alph{enumi}}
\item\label{it:MTa} for every $v_1,v_2\in T$ there is $w\in T$ such that $v_1\Lambda
    w\ne\emptyset$ and $v_2\Lambda w\ne\emptyset$,
\item\label{it:MTb} for every $v\in T$ and $1\le i\le k$ there exists $\lambda\in
    v\Lambda^{e_i}$ such that $s(\lambda)\in T$, and
\item\label{it:MTc} for $w\in T$ and $v\in\Lambda^0$ with $v\Lambda w\ne\emptyset$ we
    have $v\in T$.
\end{enumerate}
\end{dfn}

In \cite{DavidsonYang:CJM09}, Davidson and Yang comprehensively analyse aperiodicity for
single-vertex $2$-graphs. The following results substantially generalise this analysis,
but the fundamental idea behind them is due to Davidson-Yang.

Let $\Lambda$ be a row-finite $k$-graph with no sources such that $\Lambda^0$ is a
maximal tail. We define a relation on $\Lambda$ by
\[
\mu \sim \nu \quad\text{ if and only if }\quad s(\mu) = s(\nu)\text{ and }
    \mu x = \nu x\text{ for all $x \in s(\mu)\Lambda^\infty$.}
\]
This is an equivalence relation on $\Lambda$ which respects range, source and composition.
Thus $\Lambda/\negthickspace\sim$ is a category with respect to $r([\lambda]) = [r(\lambda)]$,
$s([\lambda]) = [s(\lambda)]$ and $[\lambda][\mu] = [\lambda\mu]$. Define $\Per(\Lambda)
\subseteq \ZZ^k$ by
\[
\Per(\Lambda) := \{d(\mu) - d(\nu) : \mu,\nu \in \Lambda\text{ and }\mu \sim \nu\},
\]
and define
\begin{equation}\label{eq:HperDef}
\begin{split}
H_{\Per} := \big\{v \in \Lambda^0 :\text{ for all } \lambda \in v\Lambda\text{ and }& m
\in \NN^k\text{ such that }
        d(\lambda) - m \in \Per(\Lambda), \\
    &\text{ there exists } \mu \in v\Lambda^m\text{ such that }\lambda \sim \mu\big\}.
\end{split}
\end{equation}

For the next result recall that if $\Lambda$ is a $k$-graph then a subset $H$ of
$\Lambda^0$ is \emph{hereditary} if $s(H\Lambda) \subseteq H$.

\begin{thm}\label{thm:pullback iso}
Let $\Lambda$ be a row-finite $k$-graph with no sources such that $\Lambda^0$ is a
maximal tail.
\begin{enumerate}
\item\label{it:Per group} The set $\Per(\Lambda)$ is a subgroup of $\ZZ^k$.
\item\label{it:Hper} The set $H_{\Per}$ is a nonempty hereditary subset of
    $\Lambda^0$, and for all $p,q \in \NN^k$ such that $p - q \in \Per(\Lambda)$ and all
    $x \in H_{\Per}\Lambda^\infty$, we have $\sigma^p(x) = \sigma^q(x)$.
\item\label{it:theta} If $r(\lambda) \in H_{\Per}$ and $d(\lambda) - m \in
    \Per(\Lambda)$, then there is a unique $\mu \in r(\lambda)\Lambda^m$ such that
    $\mu \sim \lambda$; in particular, if $\lambda \sim \mu$ and $d(\lambda) =
    d(\mu)$, then $\lambda = \mu$.
\item\label{it:quotient P-graph} If $q_{\Per}: \ZZ^k\to \ZZ^k/\Per(\Lambda)$ is the quotient map, then the set $\Gamma := (H_{\Per}\Lambda)/\negthickspace\sim$ is a
    $q_{\Per}(\NN^k)$-graph with degree map $\widetilde{d} := q_{\Per} \circ d$.
\item\label{it:tail pullback iso} The assignment $\lambda \mapsto ([\lambda],
    d(\lambda))$ is an isomorphism of $H_{\Per}\Lambda$ onto the pullback $q_{\Per}^*
    \Gamma$.
\end{enumerate}
\end{thm}

\begin{proof}[Proof of Theorem~\ref{thm:pullback iso}(\ref{it:Per group})]
Since $\lambda \sim \lambda$ for all $\lambda \in \Lambda$, we have $0 \in
\Per(\Lambda)$. That $\sim$ is symmetric shows that $\Per(\Lambda)$ is closed under
inverses in $\ZZ^k$. To see that it is closed under addition, suppose that $m,n \in
\Per(\Lambda)$, and fix $\mu \sim \nu$ and $\eta \sim \zeta$ such that $d(\mu) - d(\nu) =
m$ and $d(\eta) - d(\zeta) = n$. Since $\Lambda^0$ is a maximal tail, there exist $\alpha
\in s(\mu)\Lambda$ and $\beta \in s(\zeta)\Lambda$ such that $s(\alpha) = s(\beta) = v$.
We have $\mu\alpha y = \nu\alpha y$ for all $y \in v\Lambda^\infty$. Hence
\[
\sigma^{d(\mu)}(y)
    = \sigma^{d(\mu) + d(\nu\alpha)}(\nu\alpha y)
    = \sigma^{d(\nu) + d(\mu\alpha)}(\mu\alpha y)
    =\sigma^{d(\nu)}(y)
\]
for all $y \in v\Lambda^\infty$; and similarly, $\sigma^{d(\eta)}(y) =
\sigma^{d(\zeta)}(y)$ for all $y \in v\Lambda^\infty$. Now
\begin{align*}
\sigma^{d(\mu) + d(\eta)}(y)
    = \sigma^{d(\mu)}(\sigma^{d(\eta)}(y))
    &= \sigma^{d(\mu)}(\sigma^{d(\zeta)}(y))\\
    &= \sigma^{d(\zeta)}(\sigma^{d(\mu)}(y))
    = \sigma^{d(\zeta)}(\sigma^{d(\nu)}(y))
    = \sigma^{d(\nu) + d(\zeta)}(y)
\end{align*}
for all $y \in v\Lambda^\infty$. Fix $\xi \in v\Lambda^{d(\mu) + d(\eta)}$ and $\xi' \in
s(\xi)\Lambda^{d(\nu) + d(\zeta)}$. Factorise $\xi\xi' = \tau\tau'$ with $d(\tau) =
d(\xi')$ and $d(\tau') = d(\xi)$. For $z \in s(\xi')\Lambda^\infty$, we have $\xi\xi' z
\in v\Lambda^\infty$, and hence
\[
\xi' z
    = \sigma^{d(\xi)}(\xi\xi'z)
    = \sigma^{d(\mu) + d(\eta)}(\xi\xi'z)
    = \sigma^{d(\nu) + d(\zeta)}(\tau\tau'z)
    = \tau' z.
\]
It follows that $\xi' \sim \tau'$ and hence that $m + n = d(\tau') - d(\xi')$ belongs to
$\Per(\Lambda)$.
\end{proof}

To prove Theorem~\ref{thm:pullback iso}(\ref{it:Hper}), we need some technical results.

\begin{lem}\label{lem:ordering and addition}
Let $\Lambda$ be a row-finite $k$-graph with no sources such that $\Lambda^0$ is a
maximal tail. For $v \in \Lambda^0$, let $\Sigma_v := \{(p,q) \in \NN^k \times \NN^k :
\sigma^p(x) = \sigma^q(x)\text{ for all } x \in v\Lambda^\infty\}$.
\begin{enumerate}
\item\label{it:ordering} For each $\lambda \in \Lambda$, we have $\Sigma_{r(\lambda)}
    \subseteq \Sigma_{s(\lambda)}$.
\item\label{it:equiv rel} The relation $\Sigma_v$ is an equivalence relation on
    $\NN^k$ for each $v \in \Lambda^0$.
\item\label{it:addition} The set $\Sigma_v$ is a sub-monoid of $\NN^k \times \NN^k$
    for each $v \in \Lambda^0$.
\end{enumerate}
\end{lem}
\begin{proof}
(\ref{it:ordering}) Fix $\lambda \in \Lambda$ and $(p,q) \in \Sigma_{r(\lambda)}$. For $x
\in s(\lambda)\Lambda^\infty$, we have
\[
\sigma^p(x)
    = \sigma^{d(\lambda)}(\sigma^p(\lambda x))
    = \sigma^{d(\lambda)}(\sigma^q(\lambda x))
    = \sigma^q(x),
\]
so $(p,q) \in \Sigma_{s(\lambda)}$.

(\ref{it:equiv rel}) It is routine to check that $\Sigma_v$ is reflexive, symmetric and transitive.

(\ref{it:addition}) Fix $v \in \Lambda^0$. Part~(\ref{it:equiv rel}) implies that $(0,0)
\in \Sigma_v$. Suppose that $(p,q), (p',q') \in \Sigma_v$, and fix $x \in
v\Lambda^\infty$. Then
\[
\sigma^{p+p'}(x)
    = \sigma^p(\sigma^{p'}(x))
    = \sigma^p(\sigma^{q'}(x)).
\]
Part~(\ref{it:ordering}) implies that $(p,q) \in \Sigma_{s(x(0,q'))} =
\Sigma_{r(\sigma^{q'}(x))}$, and hence $\sigma^p(\sigma^{q'}(x))=
\sigma^q(\sigma^{q'}(x)) = \sigma^{q + q'}(x)$. Hence $(p+p', q+q') \in \Sigma_v$.
\end{proof}

\begin{prop}\label{prp:finitely generated}
Let $\Lambda$ be a row-finite $k$-graph with no sources such that $\Lambda^0$ is a
maximal tail. Let $\Sigma_\Lambda = \bigcup_{v \in \Lambda^0} \Sigma_v$. Then
$\Sigma_\Lambda$ is an equivalence relation on $\NN^k$ and a submonoid of $\NN^k \times
\NN^k$. The set $\Sigma^{\min}_\Lambda$ of minimal elements of $\Sigma_\Lambda \setminus
\{0\}$ (under the usual ordering on $\NN^k \times \NN^k$) is finite and generates
$\Sigma_\Lambda$ as
a monoid; and $(\Sigma_\Lambda - \Sigma_\Lambda) \cap (\NN^k \times \NN^k) =
\Sigma_\Lambda$.
\end{prop}
\begin{proof}
That $\Sigma_\Lambda$ is reflexive and symmetric follows from the same properties of the
$\Sigma_v$. If $(p,q), (q,r) \in \Sigma_\Lambda$ then there exist $v,w$ such that $(p,q)
\in \Sigma_v$ and $(q,r) \in \Sigma_w$. Since $\Lambda^0$ is a maximal tail, there exists
$u \in \Lambda^0$ such that $v\Lambda u$ and $w \Lambda u$ are nonempty.
Lemma~\ref{lem:ordering and addition}(\ref{it:ordering}) then implies that $(p,q), (q,r)
\in \Sigma_u$. Lemma~\ref{lem:ordering and addition}(\ref{it:equiv rel}) therefore
implies that $(p,r) \in \Sigma_u \subseteq \Sigma_\Lambda$.

We show that $\Sigma_\Lambda$ is a submonoid of $\NN^k \times \NN^k$. That each
$\Sigma_v$ is a monoid  gives $(0,0) \in \Sigma_\Lambda$, so it suffices to show that
$\Sigma_\Lambda$ is closed under addition. Suppose that $(p,q)$ and $(r,s)$ belong to
$\Sigma_\Lambda$. As above, there exists $u \in \Lambda^0$ such that $(p,q), (r,s) \in
\Sigma_u$. Lemma~\ref{lem:ordering and addition}(\ref{it:addition}) then implies that
$(p,q) + (r,s) \in \Sigma_u \subseteq \Sigma_\Lambda$.

Let $G := \{p-q : p,q \in \Sigma_\Lambda\ \subseteq \ZZ^k \times \ZZ^k$. We must show
that $G
\cap (\NN^k \times \NN^k) = \Sigma_\Lambda$. Since $\Sigma_\Lambda$ is a monoid, the
containment $\supseteq$ is clear. For the reverse containment, suppose that $(p,q), (r,s)
\in \Sigma_\Lambda$ and that $p-r$ and $q-s$ belong to $\NN^k$; we must show that
$(p,q)-(r,s) \in \Sigma_\Lambda$. As above, there exists $u \in \Lambda^0$ such that
$(p,q)$ and $(r,s)$ both belong to $\Sigma_u$. Fix $\lambda \in u\Lambda^{r+s}$, and let
$w  = s(\lambda)$. Lemma~\ref{lem:ordering and addition}(\ref{it:equiv rel}) implies that
$(s,r)$ belongs to $\Sigma_u$ and then that $(p+s, q+r) \in \Sigma_u$. So for $x \in
w\Lambda^\infty$, we have
\[
\sigma^{p-r}(x)
    = \sigma^{p-r}(\sigma^{r+s}(\lambda x))
    = \sigma^{p+s}(\lambda x)
    = \sigma^{q+r}(\lambda x)
    = \sigma^{q-s}(\sigma^{r+s}(\lambda x))
    = \sigma^{q-s}(x).
\]
Thus $(p,q)-(r,s) \in \Sigma_w \subseteq \Sigma_\Lambda$.

The set $\Sigma^{\min}_\Lambda$ is completely unarranged in the sense that no two
distinct elements of $\Sigma^{\min}_\Lambda$ are comparable under $\le$. Hence
$\Sigma^{\min}_\Lambda$ is finite by Dickson's Theorem (see
\cite[Theorem~5.1]{RosalesSanchez:monoids}). An induction (see, for example,
\cite[Proposition~8.1]{RosalesSanchez:monoids}) shows that it generates $\Sigma_\Lambda$
as a monoid.
\end{proof}

\begin{lem}\label{lem:Sigma translation}
Let $\Lambda$ be a row-finite $k$-graph with no sources and suppose that $\Lambda^0$ is a
maximal tail. Suppose that $p,q\in\NN^k$ and $n \in \ZZ^k$ satisfy $(p-n, q-n) \in
\Sigma_\Lambda$. Then $(p, q) \in \Sigma_\Lambda$.
\end{lem}
\begin{proof}
Let $n_+=n\vee 0$ and $n_-=(-n)\vee 0$. Then $n_+,n_-\in \NN^k$ and $n=n_+-n_-$. Putting
$p' := p-n$ and $q' := q-n$, we have
\[
(p, q)=((p',q')+(n_+,n_+))-(n_-,n_-)\in (\Sigma_\Lambda-\Sigma_\Lambda)\cap(\NN^k\times \NN^k).
\]
Proposition~\ref{prp:finitely generated} therefore implies that $(p,q) \in
\Sigma_\Lambda$.
\end{proof}

\begin{lem}\label{lem:Per vs Sigma}
Let $\Lambda$ be a row-finite $k$-graph with no sources such that $\Lambda^0$ is a
maximal tail. For $p,q \in \NN^k$, we have $(p,q) \in \Sigma_\Lambda$ if and only if $p -
q \in \Per(\Lambda)$.
\end{lem}
\begin{proof}
Suppose that $p,q \in \NN^k$ satisfy $p - q \in \Per(\Lambda)$. Fix $\mu \sim \nu$ in
$\Lambda$ such that $p-q = d(\mu) - d(\nu)$. Let $p' = d(\mu)$ and $q' = d(\nu)$. Since
$\mu \sim \nu$, for $x \in s(\mu)\Lambda^\infty$ we have
\[
\sigma^{p'}(x) = \sigma^{p'+q'}(\nu x) = \sigma^{p'+q'}(\mu x) = \sigma^{q'}(x).
\]
So $(p',q') \in \Sigma_{s(\mu)} \subseteq \Sigma_\Lambda$. We have $p - q = p' - q'$, and
hence $p - p' = q - q'$. Thus $n := p - p' \in \ZZ^k$ and $p,q \in \NN^k$ satisfy $(p-n,
q-n) \in \Sigma_\Lambda$. Thus Lemma~\ref{lem:Sigma translation} implies that $(p,q) \in
\Sigma_\Lambda$.

Now suppose that $(p,q) \in \Sigma_\Lambda$. Fix $v$ such that $(p,q) \in \Sigma_v$. Fix
$\lambda \in v\Lambda^{p+q}$. Factorise $\lambda = \mu\alpha = \nu\beta$ with $d(\mu) =
d(\beta) = p$ and $d(\nu) = d(\alpha) = q$. For $x \in s(\alpha)\Lambda^\infty$, we have
\[
\alpha x
    = \sigma^p(\mu\alpha x)
    = \sigma^q(\nu\beta x)
    = \beta x.
\]
Hence $\alpha \sim \beta$ and so $p - q = d(\beta) - d(\alpha) \in \Per(\Lambda)$.
\end{proof}

\begin{proof}[Proof of Theorem~\ref{thm:pullback iso}(\ref{it:Hper})]
We begin by showing that $H_{\Per}$ is nonempty. By Proposition~\ref{prp:finitely
generated} the finite set $\Sigma^{\min}_\Lambda$ generates $\Sigma_\Lambda$ as a monoid.
For each $(p,q) \in \Sigma^{\min}_\Lambda$ there exists $w \in \Lambda^0$ such that $(p,
q) \in \Sigma_{w}$. Let $|\Sigma^{\min}_\Lambda|$ be the cardinality of
$\Sigma^{\min}_\Lambda$. By $|\Sigma^{\min}_\Lambda|$ applications of
condition~(\ref{it:MTa}) for the maximal tail $\Lambda^0$, there exists $w \in \Lambda^0$
such that $\Sigma^{\min}_\Lambda \subseteq \Sigma_w$. Since $\Sigma_w$ is a monoid and
$\Sigma_w\subseteq \Sigma_\Lambda$, it follows that $\Sigma_w = \Sigma_\Lambda$. Let $N
:= \bigvee\{p \vee q : (p,q) \in \Sigma^{\min}_\Lambda\}$, and fix $\eta \in w\Lambda^N$.
We claim that $v := s(\eta)$ belongs to $H_{\Per}$. To see this, fix $\lambda \in
v\Lambda$ and $m \in \NN^k$ such that $d(\lambda) - m \in \Per(\Lambda)$.
Lemma~\ref{lem:Per vs Sigma} implies that $(d(\lambda), m) \in \Sigma_\Lambda$. So
Proposition~\ref{prp:finitely generated} implies that there exist elements $(p_1, q_1),
\dots, (p_l, q_l) \in \Sigma^{\min}_\Lambda$ such that
\[
(d(\lambda), m) = \sum^l_{i=1} (p_i, q_i).
\]
We will argue by induction that for $0 \le j \le l$ there exists $\lambda_j \in v\Lambda$
such that $\lambda_j \sim \lambda$ and $d(\lambda_j) = d(\lambda) + \sum^j_{i=1} (q_i -
p_i)$. Putting $\lambda_0 := \lambda$ establishes a base case. Now suppose that we have
$\lambda_{j-1}$ with the desired properties. Recall that $r(\lambda) = s(\eta)$ where
$\eta \in \Lambda^N$ and $\Sigma_{r(\eta)} = \Sigma_\Lambda$. Let $\mu := \eta(N - q_j,
N)$. By construction of $\lambda_{j-1}$ we have
\[\textstyle
d(\mu\lambda_{j-1}) = q_j + d(\lambda) + \sum^{j-1}_{i=1} (q_i - p_i) \ge p_j,
\]
so $\lambda_j := (\mu\lambda_{j-1})(p_j, d(\mu\lambda_{j-1}))$ satisfies
\[\textstyle
d(\lambda_j) = d(\mu\lambda_{j-1}) - p_j = d(\lambda) + \sum^j_{i=1} (q_i - p_i).
\]
Fix $x \in s(\lambda_j)\Lambda^\infty$. We have
\[
\lambda_j x
    = (\mu\lambda_{j-1})(p_j, d(\mu\lambda_{j-1})) x
    = \sigma^{p_j}(\mu\lambda_{j-1} x).
\]
Lemma~\ref{lem:ordering and addition}(\ref{it:ordering}) implies that $\Sigma_\Lambda =
\Sigma_{r(\eta)} \subseteq \Sigma_{r(\nu)} \subseteq \Sigma_\Lambda$, giving equality
throughout. Thus $(p_j, q_j) \in \Sigma_{r(\nu)}$, and so
\[
\sigma^{p_j}(\mu\lambda_{j-1} x)
    = \sigma^{q_j}(\mu\lambda_{j-1} x)
    = \lambda_{j-1} x.
\]
So $\lambda_j \sim \lambda_{j-1}$. Since $\lambda_{j-1} \sim \lambda$ by the inductive
hypothesis, $\lambda_j \sim \lambda$.

Now $\mu := \lambda_l$ satisfies $\mu \in v\Lambda^m$ and $\lambda \sim \mu$. Hence $v
\in H_{\Per}$ as claimed.

We show that $H_{\Per}$ is hereditary. Suppose that $r(\alpha) \in H_{\Per}$. Fix
$\lambda \in s(\alpha)\Lambda$ and $m \in \NN^k$ such that $d(\lambda) - m \in
\Per(\Lambda)$. Then $d(\alpha\lambda) - (m + d(\alpha)) \in \Per(\Lambda)$. Hence there
exists $\zeta \in r(\alpha)\Lambda$ such that $d(\zeta) = m + d(\alpha)$ and $\zeta \sim
\alpha\lambda$. Since $\zeta x = \alpha\lambda x$ for all $x\in
s(\lambda)\Lambda^\infty$, and since and $d(\zeta) > d(\alpha)$, we have $\zeta(0,
d(\alpha)) = \alpha$. Thus $\zeta = \alpha \mu$ for some $\mu \in s(\alpha)\Lambda^m$.
For $x \in s(\lambda)\Lambda$,
\[
\lambda x = \sigma^{d(\alpha)}(\alpha\lambda x)
    = \sigma^{d(\alpha)}(\zeta x)
    = \sigma^{d(\alpha)}(\alpha\mu x)
    = \mu x.
\]
So $\lambda \sim \mu$, whence $s(\alpha) \in H_{\Per}$.

To see that $\sigma^p(x) = \sigma^q(x)$ whenever $p-q \in \Per(\Lambda)$ and $x \in
H_{\Per}\Lambda^\infty$, fix such $p, q, x$. Let $\lambda = x(0,p)$. Then $x = \lambda
\sigma^p(x)$. Since $r(x) \in H_{\Per}$, there exists $\mu \in r(x)\Lambda^q$ such that
$\mu \sim \lambda$. We then have
\[
\sigma^p(x) = \sigma^q(\mu\sigma^p(x)) = \sigma^q(\lambda \sigma^p(x)) = \sigma^q(x).\qedhere
\]
\end{proof}

\begin{proof}[Proof of Theorem~\ref{thm:pullback iso}(\ref{it:theta})]
Fix $\lambda \in H_{\Per} \Lambda$ and $m \in \NN^k$ such that $d(\lambda) - m \in
\Per(\Lambda)$. Suppose that $\mu,\nu \in\Lambda^m$ satisfy $\mu \sim \lambda \sim \nu$.
Since $\sim$ is an equivalence relation, we have $\mu \sim \nu$. So it suffices to show
that if $\mu \sim \nu$ and $d(\mu) = d(\nu)$, then $\mu = \nu$. Fix $x \in
s(\mu)\Lambda^\infty$ and observe that $\mu \sim \nu$ implies $\mu x = \nu x$ and so $\mu
= (\mu x)(0,m) = (\nu x)(0,m) = \nu$.
\end{proof}

\begin{proof}[Proof of Theorem~\ref{thm:pullback iso}(\ref{it:quotient P-graph})]
We showed that $\Gamma$ is a category at the beginning of this section.
 Since $d$ is a functor, $\widetilde{d}:=q_{\Per}\circ d$ is
also a functor. We must show that $(\Gamma,\widetilde{d})$ has the unique factorisation
property.

Let $q=q_{\Per(\Lambda)}:\ZZ^k\to \ZZ^k/\Per(\Lambda)$. Fix $\lambda\in\Lambda$ with
$r(\lambda)\in H_{\Per}$ and $m,n\in\NN^k$ such that
$\widetilde{d}([\lambda])=q(m)+q(n)$. We must show that there is a unique pair
$\alpha,\beta\in\Gamma$ such that $\widetilde{d}(\alpha)=q(m)$,
$\widetilde{d}(\beta)=q(n)$ and $\alpha\beta=[\lambda]$. We begin by showing that such a
pair exists. Since $d(\lambda) - (m+n) \in \Per(\Lambda)$ and $r(\lambda) \in H_{\Per}$,
there exists $\mu \in r(\lambda) \Lambda^{m+n}$ such that $\mu \sim \lambda$. Let $\alpha
:= [\mu(0,m)]$ and $\beta := [\mu(m, m+n)]$. Then $\alpha\beta = [\mu(0,m)\mu(m, m+n)] =
[\mu] = [\lambda]$ because $\mu \sim \lambda$; and $d(\alpha) = q(m)$ and $d(\beta) =
q(n)$ by construction.

For uniqueness, suppose that $\mu_1,\nu_1,\mu_2,\nu_2\in H_{\Per}\Lambda$, satisfy
$r(\nu_i) = s(\mu_i)$, $q(d(\mu_i)) = q(m)$ and $q(d(\nu_i)) = q(n)$. Suppose further
that $\mu_1\nu_1\sim \lambda\sim\mu_2\nu_2$. We must show that $\mu_1\sim\mu_2$ and
$\nu_1\sim\nu_2$. Fix $x \in s(\nu_1)\Lambda^\infty$. Then
\[
\nu_1 x
    = \sigma^{d(\mu_1)}(\mu_1\nu_1 x)
    = \sigma^{d(\mu_1)}(\mu_2\nu_2 x).
\]
We have $q(d(\mu_1)) = q(m) = q(d(\mu_2))$, and so $d(\mu_1) - d(\mu_2) \in
\Per(\Lambda)$. Hence Theorem~\ref{thm:pullback iso}(\ref{it:Hper}) implies that
$\sigma^{d(\mu_1)}(y)=\sigma^{d(\mu_2)}(y)$ for all $y\in r(\mu_2)\Lambda^\infty$, and in particular
\[
\sigma^{d(\mu_1)}(\mu_2\nu_2 x)
    = \sigma^{d(\mu_2)}(\mu_2\nu_2 x)
    = \nu_2 x.
\]
Thus $\nu_1 \sim \nu_2$. Since $d(\mu_1) - d(\mu_2) \in \Per(\Lambda)$ and since
$r(\mu_1) \in H_{\Per}$, there exists a unique $\zeta \in r(\mu_1)\Lambda^{d(\mu_2)}$
such that $\mu_1 \sim \zeta$. Fix $x \in s(\lambda)\Lambda^\infty$. We have
\[
\zeta = (\zeta\nu_1 x)(0, d(\mu_2))
    = (\mu_1 \nu_1 x)(0, d(\mu_2))
    = (\mu_2\nu_2 x)(0,d(\mu_2))
    = \mu_2.\qedhere
\]
\end{proof}

\begin{proof}[Proof of Theorem~\ref{thm:pullback iso}(\ref{it:tail pullback iso})]
It is routine to check that $\lambda\mapsto ([\lambda],d(\lambda))$ is a $k$-graph
morphism from $H_{\Per}\Lambda$ to $q_{\PER}^*(\Gamma)$. This $k$-graph morphism is
injective by part~(\ref{it:theta}). Fix $\xi\in H_{\Per}\Lambda$ and $m\in\NN^k$ such
that $m-d(\xi)\in\PER$. Then there exists $\eta \in r(\xi)\Lambda^m$ such that $\eta \sim
\xi$. Thus $([\eta],d(\eta))=([\xi],m)$. So $\lambda\mapsto ([\lambda],d(\lambda))$ is
surjective and hence an isomorphism.
\end{proof}

\section{The primitive ideals of a \texorpdfstring{$k$}{k}-graph algebra}\label{sec:all prim ideals}

In this section we prove our main theorem, giving a complete listing of the primitive
ideals in the $C^*$-algebra of a row-finite $k$-graph with no sources.

\begin{lem} \label{lem:corner}
Let $\Lambda$ be a row-finite $k$-graph such that $\Lambda^0$  is a maximal tail. Let $H
= H_{\Per}$ be the hereditary set~\eqref{eq:HperDef}, and let $\Gamma =
H\Lambda/\negthickspace\sim$. Let $q = q_{\Per} : \ZZ^k \to \ZZ^k/\Per(\Lambda)$ be the
quotient map, and let $P := q(\NN^k)$. Then
\begin{enumerate}
\item\label{it:exists homo} there is a unique $*$-homomorphism $\phi : C^*(q^*\Gamma)
    \to C^*(\Lambda)$ such that $\phi(s_{([\lambda],d(\lambda))})=s_\lambda$ for
    $\lambda\in H\Lambda$;
\item\label{it:corner iso} the series $\sum_{v \in H} s_v$ converges strictly to a
    projection $P_H$ in $\Mm C^*(\Lambda)$, and $\phi$ is an isomorphism of
    $C^*(q^*\Gamma)$ onto $P_H C^*(\Lambda) P_H$; and
\item\label{it:corner in ideal} the corner $P_H C^*(\Lambda) P_H$ is a full corner of
    the ideal $I_H = \clsp\{s_\lambda s^*_\mu : \lambda,\mu \in \Lambda H\}$.
\end{enumerate}
\end{lem}
\begin{proof}
Theorem~\ref{thm:pullback iso}(\ref{it:tail pullback iso}) implies that $\lambda\mapsto
([\lambda],d(\lambda))$ is a $k$-graph isomorphism between $H\Lambda$ and $f^*\Gamma$.
Statement~(\ref{it:exists homo}) then follows from the universal property of
$C^*(q^*\Gamma)$.

The argument of \cite[Lemma~1.1]{BatesPaskEtAl:NYJM00} shows that $\sum_{v \in H} s_v$
converges strictly to a multiplier projection $P_H$ such that $P_H s_\mu s^*_\nu =
\chi_H(r(\mu)) s_\mu s^*_\nu$. We have $P_H C^*(\Lambda) P_H = \clsp\{s_\lambda s^*_\mu
: \lambda,\mu \in H\Lambda\}$, so the image of $\phi$ is $P_H C^*(\Lambda) P_H$.
Proposition~\ref{prp:Pgraph giut} implies that $\phi$ is injective,
giving~(\ref{it:corner iso}).

For~(\ref{it:corner in ideal}), observe that $I_H$ is the ideal generated by $\{s_v:v\in
H\}$ and so is generated as an ideal of $C^*(\Lambda)$ by $P_H$. Hence $P_H$ determines a
full corner in $I_H$.
\end{proof}

\begin{lem}\label{lem:infinite-path}
Let $\Lambda$ be a row-finite $k$-graph with no sources and let $T$ be a maximal tail of
$\Lambda$. Then there is an infinite path $x\in(\Lambda T)^{\infty}$ which is cofinal in
$T$ in the sense that for each $v\in T$, there exists $n\in\NN^k$ such that $v\Lambda
x(n)\neq \emptyset$.
\end{lem}
\begin{proof}
Let $\{v_j\}^{\infty}_{j=0}$ be a listing of $T$. Condition~(\ref{it:MTa}) for a maximal
tail implies that there exists $\alpha_1\in v_0 T$ such that $v_1\Lambda
s(\alpha_1)\neq\emptyset$. Applying the same condition again gives $\alpha_2\in
s(\alpha_1) T$ such that $v_2\Lambda s(\alpha_2)\neq\emptyset$. Inductively, we construct
$\{\alpha_j\}^{\infty}_{j=1}$ such that $r(\alpha_j)=s(\alpha_{j-1})$ and $v_j\Lambda
s(\alpha_j)\neq\emptyset$ for all $j$. Define $\beta_0=v_0$ and
$\beta_i=\beta_{i-1}\alpha_i$ for $i\ge 1$. Then $v_i\Lambda s(\beta_j)\neq \emptyset$
whenever $i \le j$, and $d(\beta_j)\rightarrow (\infty,\dots,\infty)$ as
$j\rightarrow\infty$. The argument of \cite[Remarks~2.2]{KP1} shows that there is a
unique $x\in (\Lambda T)^{\infty}$ such that $x(0,d(\beta_i))=\beta_i$ for all $i$. This
$x$ is cofinal by construction.
\end{proof}

In the following, given an infinite path $x$ in a $k$-graph $\Lambda$, we write $[x]$ for
the shift-tail equivalence class $\{\lambda\sigma^m(x) : m \in \NN^k, \lambda \in \Lambda
x(m)\}$.

\begin{thm}\label{thm:primitive ideal}
Let $\Lambda$ be a row-finite $k$-graph with no sources.
\begin{enumerate}
\item\label{it:T,gamma->I} Let $T$ be a maximal tail of $\Lambda$ and $\gamma$ a
    character of $\Per(\Lambda T)$. Suppose that $z \in \TT^k$ satisfies $z^{m-n} =
    \gamma(m-n)$ for all $m-n \in \Per(\Lambda T)$. Suppose that $x \in (\Lambda
    T)^\infty$ is cofinal in $T$. Then there is an irreducible representation
    $\pi_{[x], z} : C^*(\Lambda) \to \Bb(\ell^2([x]))$ determined by
    \begin{equation}\label{eq:irrep formula}
    \pi_{[x], z}(s_{\lambda})\xi_y
        = \begin{cases}
            z^{d(\lambda)}\xi_{\lambda y} & \text{ if $y(0)=s(\lambda)$}\\
            0 & \text{ otherwise.}
        \end{cases}
    \end{equation}
\item\label{it:I->T,gamma} Let $I$ be a primitive ideal in $C^*(\Lambda)$. Then $T =
    \{v \in \Lambda^0 : s_v\notin I\}$ is a maximal tail. Let $H = H_{\Per}(\Lambda
    T)$. There is a unique character $\gamma$ of $\Per(\Lambda T)$ such that
    $s_\mu-\gamma(d(\mu)-d(\nu))s_\nu\in I$ whenever $\mu,\nu \in H\Lambda T$ and
    $\mu \sim_{\Lambda T} \nu$. Suppose that $x \in (\Lambda T)^\infty$ is cofinal in
    $\Lambda T$ and that $z \in \TT^k$ satisfies $z^{m-n} = \gamma(m - n)$ for all
    $m-n \in \Per(\Lambda T)$. Then $I = \ker\pi_{[x], z}$.
\end{enumerate}
\end{thm}
\begin{proof}
(\ref{it:T,gamma->I}) It is routine to check that the operators $\{\pi_{[x],
z}(s_\lambda) : \lambda \in \Lambda\}$ defined by~\eqref{eq:irrep formula} constitute a
Cuntz-Krieger $\Lambda$-family. The universal property of $C^*(\Lambda)$ then yields a
homomorphism $\pi := \pi_{[x], z} : C^*(\Lambda) \to \Bb(\ell^2([x]))$
extending~\eqref{eq:irrep formula}.

Fix $y\in [x]$. For $w\in [x]$ and $n \in \NN^k$,
\[
\pi(s_{y(0,n)}s^*_{y(0,n)})\delta_w
    = \begin{cases}
        \delta_w &\text{if $w(0,n)=y(0,n)$}\\
        0 &\text{otherwise.}
    \end{cases}
\]
Hence $\lim_{n\rightarrow \infty} \pi(s_{y(0,n)} s^*_{y(0,n)}) \delta_w =
\theta_{\delta_y, \delta_y} \delta_w$. Thus the net $(\pi(s_{y(0,n)}
s^*_{y(0,n)}))_{n\in\NN^k}$ converges strongly to $\theta_{y,y}$. We claim that the
strong closure of the image of $\pi$ contains $\Kk(\ell^2([x]))$. Fix $w,y\in[x]$ and fix
$p_i,q_i\in \NN^k$ such that $\sigma^{p_1}(w)=\sigma^{q_1}(x)$ and
$\sigma^{p_2}(y)=\sigma^{q_2}(x)$. Let $q=q_1\vee q_2$, $\alpha=w(0,p_1+(q-q_1))$ and
$\beta=y(0,p_2+(q-q_2))$. Then $w=\alpha\sigma^q(x)$ and $y=\beta\sigma^q(x)$. We have
$\pi(s_{\alpha x(q,q+n)}s^*_{\beta x(q,q+n)}) = \pi(s_{\alpha}s^*_{\beta})
\pi(s_{y(0,d(\beta)+n)} s^*_{y(0,d(\beta)+n)})$ which converges strongly to
$\pi(s_{\alpha} s^*_{\beta})\theta_{y,y} = \theta_{w,y}$. Thus, the strong closure of the
image of $\pi$ contains $\Kk(\ell^2([x]))$ as claimed. Hence $\pi$ is irreducible.

(\ref{it:I->T,gamma}) Since $I$ is a primitive ideal in $C^*(\Lambda)$, there is an
irreducible representation $\pi : C^*(\Lambda)\to \mathcal{B}(\Hh)$ such that
$\ker\pi=I$. Let $K=\{v\in\Lambda^0 : \pi(s_v)=0\}$. Then
\cite[Theorem~5.2(a)]{RaeburnSimsEtAl:PEMS03} implies that $K$ is hereditary and
saturated. So the set
\[
    T=\{v\in\Lambda^0: s_v \notin I\}=\{v\in\Lambda^0:\pi(s_v)\neq 0\}= \Lambda^0 \setminus K.
\]
satisfies conditions (\ref{it:MTb})~and~(\ref{it:MTc}) of a maximal tail. To
establish~(\ref{it:MTa}), fix $v,w\in T$ and $h \in p_v \Hh \setminus \{0\}$. Since $\pi$
is irreducible, $h$ is cyclic for $\pi$. So there exists $a \in C^*(\Lambda)$ such that
$\pi(a)h \in \pi(p_w) \Hh\setminus \{0\}$; that is, $0 \not= \pi(p_w) \pi(a) h = \pi(p_w
a p_v)h$. In particular, $p_w a p_v \not= 0$. Since $C^*(\Lambda) = \clsp\{s_\alpha
s^*_\beta : s(\alpha) = s(\beta)\}$ we may approximate $a$ by a linear combination of
such $s_\alpha s^*_\beta$. Thus there exist $\alpha, \beta$ with $s(\alpha) = s(\beta)$
such that $p_w s_\alpha s^*_\beta p_v \not= 0$. Hence $u := s(\alpha) = s(\beta)$
satisfies $v\Lambda u, w \Lambda u \not= \emptyset$, and so $T$ satisfies~(\ref{it:MTa}).

Let $H = H_{\Per}(\Lambda T)$, and let $I_H$ be the ideal of $C^*(\Lambda)$ generated by
$\{s_v : v \in H\}$. Since $\pi|_{I_H}$ is nonzero, Theorem~1.3.4 of
\cite{Arveson:invitation} shows that $\pi_H := \pi|_{I_H}$ is an irreducible
representation. Lemma~\ref{lem:corner}(\ref{it:corner iso}) implies that $I_H P_H$ is a
Morita equivalence between $P_H C^*(\Lambda) P_H$, so $\ker(\pi_H)$ is induced from the
primitive ideal $\ker(\pi_H)|_{P_H C^*(\Lambda) P_H}$. Let $\Gamma = (H \Lambda T)/\sim$.
Then Lemma~\ref{lem:corner}(\ref{it:corner iso}) provides an isomorphism $\omega :
C^*(q^*\Gamma) \to P_H C^*(\Lambda) P_H$, and then $\pi_H \circ \omega$ is an irreducible
representation $C^*(q^*\Gamma)$. Now Theorem~\ref{thm:bijection} implies (see
Remark~\ref{rmk:unique character}) that there is a unique character $\gamma$ of
$\Per(\Lambda T)$ such that $s_\mu-\gamma(d(\mu)-d(\nu))s_\nu\in I$ whenever $\mu,\nu \in
H\Lambda T$ and $\mu \sim_{\Lambda T} \nu$. It is routine to check that if $x$ is cofinal
in $\Lambda T$ and $z^{m-n} = \gamma(m-n)$ for $m-n \in \Per(\Lambda T)$, then
$s_\mu-\gamma(d(\mu)-d(\nu))s_\nu \in \ker \pi_{[x], z}$ whenever $\mu,\nu \in H\Lambda
T$ and $\mu \sim_{\Lambda T} \nu$.

We claim that $\ker(\pi_H|_{P_H I_H P_H}) = \ker(\pi_{[x],z}|_{P_H I_H P_H})$. Each of
$\pi \circ \omega$ and $\pi_{[x],z} \circ \omega$ is a representation of $C^*(q^*\Gamma)$
which vanishes on the generators of the ideal $I_z$ of Lemma~\ref{quotient iso} and is
nonzero on every vertex projection $s_{(v,0)} \in C^*(q^*\Gamma)$. So $\pi \circ \omega$
and $\pi_{[x],z} \circ \omega$ descend to representations $\rho$ and $\rho_{[x],z}$ of
$C^*(q^*\Gamma)/I_z$. Lemma~\ref{quotient iso} gives an isomorphism $\theta : C^*(\Gamma)
\to C^*(q^*\Gamma)/I_z$, and the representations $\rho\circ\theta$ and $\rho_{[x],z}
\circ \theta$ of $C^*(\Gamma)$ are nonzero on vertex projections. Hence
Corollary~\ref{cor:Pgraph CKUT} implies that $\rho\circ\theta$ and $\rho_{[x],z} \circ
\theta$ are both injective. Thus $\rho$ and $\rho_{[x],z}$ are both injective, which
implies that $\ker(\pi \circ \omega) = I_z = \ker(\pi_{[x],z} \circ \omega)$. Since $P_H$
is a strict limit of projections in $I_H$, we have $P_H I_H P_H = P_H C^*(\Lambda) P_H$,
and so $\omega$ is an isomorphism from $C^*(q^*\Gamma)$ to $P_H I_H P_H$. This proves the
claim.

Since $I_H P_H$ is a Morita equivalence, the claim implies that $\ker(\pi_H) =
\ker(\pi_{[x],z}|_{I_H})$. Now \cite[Proposition~2.72]{RW} applied to the adjointable
left action of $C^*(\Lambda)$ by left multiplication on the standard Hilbert module
$(I_H)_{I_H}$ implies that $\ker\pi = \ker \pi_{[x], z}$.
\end{proof}

Let $T$ be a maximal tail of $\Lambda$ and $\gamma$ a character of $\Per(\Lambda T)$.
Theorem~\ref{thm:primitive ideal} implies that if $x,y$ are both cofinal in $T$ and $w,z
\in \TT^k$ satisfy $w^{m-n} = \gamma(m - n) = z^{m-n}$ whenever $m-n \in \Per(\Lambda
T)$, then $\ker\pi_{[x],z} = \ker\pi_{[y], w}$. We define $I_{T, \gamma} :=
\ker\pi_{[x],z}$. We write $\MT(\Lambda)$ for the collection of all maximal tails in
$\Lambda^0$.

\begin{cor}\label{cor:prim ideal bijection}
Let $\Lambda$ be a row finite $k$-graph with no sources. The assignment $(T, \gamma)
\mapsto I_{T, \gamma}$ is a bijection from $\bigcup_{T \in \MT(\Lambda)} (\{T\} \times
\widehat{\Per(\Lambda T)})$ to the set of primitive ideals of $C^*(\Lambda)$.
\end{cor}
\begin{proof}
This follows immediately from Theorem~\ref{thm:primitive ideal} and Lemma~\ref{lem:infinite-path}.
\end{proof}

\begin{rmk}
It is worthwhile to point out what the above catalogue of primitive ideals says for
aperiodic maximal tails $T$. In this instance, we have $\Per(\Lambda T) = \{0\}$ so the
only character is the identity character $\gamma(0) = 1$. The Cuntz-Krieger uniqueness
theorem implies that $\pi_{[x],1}$ restricts to a faithful representation of
$C^*(H\Lambda T)$. Thus $I_{T, 1}$ is precisely the gauge-invariant primitive ideal
associated to $T$ described in \cite{KaP1}. In particular, if $\Lambda$ is strongly
aperiodic in the sense of~\cite{KaP1}, so that every maximal tail $T$ is aperiodic, then
our result recovers the listing given in \cite{KaP1}.
\end{rmk}

We close by characterising primitivity of $C^*(\Lambda)$; recall that a $C^*$-algebra $A$
is primitive if it has a faithful irreducible representation.

\begin{cor}
Let $\Lambda$ be a row-finite $k$-graph with no sources. Then $C^*(\Lambda)$ is primitive
if and only if $\Lambda^0$ is a maximal tail and $\Lambda$ is aperiodic.
\end{cor}
\begin{proof}
First suppose that $\Lambda^0$ is a maximal tail and $\Lambda$ is aperiodic.
Lemma~\ref{lem:infinite-path} implies that there is an infinite path $x$ of $\Lambda$
which is cofinal in $\Lambda^0$. Consider the irreducible representation $\pi_{[x], 1}$
of~\eqref{eq:irrep formula}. Since $\Lambda$ is a maximal tail, we have $\pi_{[x],
1}(s_v) \not= 0$ for all $v \in \Lambda^0$. Since $\Lambda$ is aperiodic, the
Cuntz-Krieger uniqueness theorem implies that $\pi_{[x],1}$ is faithful. Hence
$C^*(\Lambda)$ is primitive.

Now suppose that $C^*(\Lambda)$ is primitive. Then $\{0\}$ is a primitive ideal of
$C^*(\Lambda)$. Theorem~\ref{thm:primitive ideal}(\ref{it:I->T,gamma}) implies that $T :=
\{v \in \Lambda^0 : s_v \not\in \{0\}\}$ is a maximal tail, and that there is a character
$\chi$ of $\Per(T)$ such that $I_{T, \chi} = \{0\}$. Since the generators of
$C^*(\Lambda)$ are all nonzero, $T = \Lambda^0$, so $\Lambda^0$ is a maximal tail. To see
that $\Lambda$ is aperiodic, we must show that $\Per(\Lambda) = \{0\}$. By
Theorem~\ref{thm:pullback iso}, it suffices to show that $\mu \sim \nu$ implies $d(\mu) =
d(\nu)$. Suppose that $\mu \sim \nu$. Then $s_\mu - \chi(d(\mu) - d(\nu)) s_\nu \in I_{T,
\chi}$, and hence $s_\mu = \chi(d(\mu) - d(\nu)) s_\nu$. For $z \in \TT^k$, we then have
\[
z^{d(\mu)} s_\mu
    = \gamma_z(s_\mu)
    = \gamma_z(\chi(d(\mu) - d(\nu)) s_\nu)
    = z^{d(\nu)} \chi(d(\mu) - d(\nu)) s_\nu.
\]
So $z^{d(\mu) - d(\nu)} s_\mu = \chi(d(\mu) - d(\nu)) s_\nu$ for all $z \in \TT^k$. Since
the right-hand side is nonzero and independent of $z$, we deduce that $z^{d(\mu) -
d(\nu)}$ is constant with respect to $z$, forcing $d(\mu) = d(\nu)$.
\end{proof}

\end{document}